\documentclass[twoside, 12pt]{article} 
\newcommand{\8}{\infty}

\newcommand{\Z}{\mathbb{Z}}

\newcommand{\id}{\mathbf{1}}

\newcommand{\diag}{\operatorname{diag}}

\newcommand{\Den}{\operatorname{Den}}

\newcommand{\SL}{\operatorname{SL}}

\newcommand{\GL}{\operatorname{GL}}

\newcommand{\HH}{\mathbb{H}}
\newcommand{\R}{\mathbb{R}}
\newcommand{\A}{\mathbb{A}}
\newcommand{\Q}{\mathbb{Q}}
\newcommand{\C}{\mathbb{C}}

\newcommand{\PP}{\mathbb{P}}
\newcommand{\afrak}{\mathfrak{a}}
\newcommand{\bfrak}{\mathfrak{b}}

\newcommand{\p}{\mathfrak{p}}

\newcommand{\q}{\mathfrak{q}}
\newcommand{\g}{\mathfrak{g}}

\newcommand{\cfrak}{\mathfrak{c}}

\newcommand{\ufrak}{\mathfrak{u}}
\newcommand{\vfrak}{\mathfrak{v}}

\newcommand{\Mfrak}{\mathfrak{M}}
\newcommand{\kfrak}{\mathfrak{k}}

\newcommand{\cohom}{H}

\newcommand{\Hom}{\operatorname{Hom}}

\newcommand{\N}{\operatorname{N}}

\newcommand{\Ehat}{\widehat{E}}

\newcommand{\Ocal}{\mathscr{O}}

\newcommand{\Hcal}{\mathcal{H}}

\newcommand{\Lcal}{\mathcal{L}}

\newcommand{\Ccal}{\mathcal{C}}

\newcommand{\Cl}{\textrm{Cl}}

\newcommand{\Ical}{\mathcal{I}}

\newcommand{\Ecal}{\mathcal{E}}
\newcommand{\Dcal}{\mathcal{D}}
\newcommand{\Rcal}{\mathcal{R}}

\newcommand{\re}{\operatorname{Re}}
\newcommand{\Res}{\operatorname{res}}

\newcommand{\Mat}{\operatorname{Mat}}
\newcommand{\im}{\operatorname{Im}}

\newcommand{\fin}{f}

\newcommand{\hooklongrightarrow}{\lhook\joinrel\longrightarrow}

\newcommand\restr[2]{{
  \left.\kern-\nulldelimiterspace 
  #1 
  \vphantom{\big|} 
  \right|_{#2} 
  }}
 
\newcommand{\Ad}{\operatorname{Ad}}

\newcommand{\Gal}{\textrm{Gal}}

\newcommand{\res}{\operatorname{res}}
\newcommand{\Eis}{\operatorname{Eis}}

\newcommand{\cohomt}{\widetilde{H}}
\newcommand{\omegatilde}{\widetilde{\omega}}
\newcommand{\alg}{\operatorname{alg}}
\newcommand{\integ}{\operatorname{int}}

\newcommand{\SU}{\operatorname{SU}}
\newcommand{\U}{\operatorname{U}}
\usepackage{booktabs} 
\usepackage{array} 
\usepackage{paralist} 
\usepackage{verbatim} 
\usepackage{subfig} 
\usepackage{mathtools}
\usepackage{amssymb}
\usepackage{amsthm}
\usepackage{tikz-cd}
\usepackage{mathrsfs}
\usepackage{enumitem}
\usepackage[colorlinks, allcolors=blue]{hyperref}
\usepackage{appendix}
\usepackage{stmaryrd}
\usepackage{titlesec}
\usepackage{bm}
\usepackage{nicefrac}
\usepackage[backend=bibtex,style=alphabetic,sorting=nyvt]{biblatex}
\usepackage[nottoc,notlof,notlot]{tocbibind} 
\usepackage[titles,subfigure]{tocloft} 
\usepackage{sectsty}
\usepackage{fancyhdr} 
\usepackage{graphicx} 
\usepackage{setspace}
\usepackage[textsize=tiny]{todonotes}
\usepackage{caption}
\newtheoremstyle{mytheoremstyle} 
    {2em}                    
    {2em}                    
    {\itshape}                   
    {}                           
    {\normalsize \bfseries \scshape}                   
    {.}                          
    {0,5em}                       
    {}  
    
\theoremstyle{mytheoremstyle}

\newtheorem{thm}{Theorem}[section]
\newtheorem*{thm*}{Theorem}
\newtheorem*{que*}{Question}
\newtheorem{cor}[thm]{Corollary}
\newtheorem*{cor*}{Corollary}
\newtheorem{lem}[thm]{Lemma}
\newtheorem{prop}[thm]{Proposition}

\theoremstyle{remark}
\newtheorem{rmk}{Remark}[section]

\numberwithin{equation}{section}

\usepackage[T1]{fontenc}
\usepackage[utf8]{inputenc} 

\usepackage[margin=0.5in, top=1.5in, bottom=1.5in, footskip=1in]{geometry} 
\sectionfont{\normalfont}
\subsectionfont{\normalfont\itshape}
\subsubsectionfont{\normalfont\itshape}
\makeatletter
\renewcommand\section{\@startsection {section}{1}{\z@}%
                                   {-3.5ex \@plus -1ex \@minus -.2ex}%
                                   {5.3ex \@plus.2ex}%
                                   {\centering \normalfont\normalsize\scshape}}
\renewcommand\subsection{\@startsection{subsection}{2}{\z@}%
                                    {-0.5in \@plus -0ex \@minus -0ex}%
                                   {1.5ex \@plus.0ex}%
                                     {\normalfont\normalsize\itshape}}
\renewcommand\subsubsection{\@startsection{subsubsection}{3}{\z@}%
                                     {-3.25ex\@plus -1ex \@minus -.2ex}%
                                     {1.5ex \@plus .2ex}%
                                     {\normalfont\small\itshape}}
\makeatother

\title{\Large An upper bound on the denominator of Eisenstein classes \\ in Bianchi manifolds}
\author{\normalsize Romain Branchereau}

\begin{document}
\maketitle
\abstract{A general conjecture of Harder relates the denominator of the Eisenstein cohomology of certain locally symmetric spaces to special values of $L$-functions. In this paper we consider the locally symmetric space $Y_\Gamma=\SL_2(\Ocal) \backslash \HH_3$ associated to $\SL_2(K)$ where $K$ is an imaginary quadratic field. Berger proves a lower bound on the denominator of the Eisenstein cohomology in certain cases. In this paper, we show how results of Ito and Sczech can be used to prove an upper bound on the denominator in terms of a special value of $L$-function. When the class number of $K$ is one, we combine this result with Berger's result to obtain the exact denominator.
}
{
  \tableofcontents
}

\section{Introduction}
Let $K=\Q(\sqrt{D})$ be an imaginary quadratic field with ring of integers $\Ocal$.  If the class number of $K$ is larger than one, we suppose\footnote{This is required to define the canonical period used in the normalization of the $L$-function.} that $D \equiv 1 \mod 8$. Let $Y_\Gamma = \Gamma \backslash \HH_3$ where $\Gamma=\SL_2(\Ocal)$ and $\HH_3$ is the hyperbolic $3$-space. It is a non-compact space and let $X_\Gamma$ be its Borel-Serre compactification with boundary $\partial X_\Gamma$. The boundary has $h=\vert \Cl(K) \vert$ connected components, one at every cusp of $Y_\Gamma$, where $h$ is the class number of $K$.

The cohomology of $Y_\Gamma$ (or more general arithmetic groups) can be studied through their cuspidal and Eisenstein part; see \cite{harder1,schwermer1}. The inclusion $Y_\Gamma \hooklongrightarrow X_\Gamma$ is a homotopy equivalence, hence $H^1(Y_\Gamma;\C)$ is isomorphic to $H^1(X_\Gamma;\C)$. By restriction to the boundary we get a map
\begin{align}
    \res \colon H^1(Y_\Gamma; \C) \longrightarrow H^1(\partial X_\Gamma; \C),
\end{align}
whose kernel is the cuspidal (or interior) cohomology that we denote by $H^1_!(Y_\Gamma;\C)$. It can be identified with the image of the compactly supported cohomology $H^1_c(Y_\Gamma;\C)$ inside $H^1(Y_\Gamma;\C)$. The Eisenstein cohomology $H_{\Eis}^1(Y_\Gamma;\C)$ is the image of $ H^1(\partial X_\Gamma; \C)$ by Harder's Eisenstein map
\begin{align} \label{eismapintro}
    \Eis \colon H^1(\partial X_\Gamma;\C) \longrightarrow H^1(Y_\Gamma; \C).
\end{align} 
It is a complement to the cuspidal cohomology, so that we have a splitting
\begin{align}\label{splitting}
    H^1(Y_\Gamma;\C) = H^1_!(Y_\Gamma;\C) \oplus H^1_{\Eis}(Y_\Gamma;\C).
\end{align}

Let $H=K(j(\Ocal))$ be the Hilbert class field of $K$. Let $\chi$ be an unramified Hecke character of infinity type $(-2,0)$ {\em i.e.} that satisfies $\chi((\alpha))=\alpha^{-2}$ on principal ideals. Let $F/H$ be a Galois extension of $H$ all the values of $\chi$ and the $h$-th roots of unity. By the work of Harder \cite{harder1}, we know that the Eisenstein map is rational, in the sense that it preserves the natural $F$-structures
\begin{align}
    \Eis \colon H^1(\partial X_\Gamma;F) \longrightarrow H^1(Y_\Gamma; F).
\end{align}
It is natural to ask how the integral structures behave. Let $\cohomt_1(Y_\Gamma;{\Ocal_F})$ be the free part of the homology with coefficients in the ring of integers ${\Ocal_F}$ of $F$, and let $\cohomt^1(Y_\Gamma;{\Ocal_F})$ be its dual with respect to the natural pairing between homology and cohomology. It is the cohomology of differential $1$-forms whose integral along any integral homology class is in ${\Ocal_F}$.
It is known that in general the Eisenstein map is not integral, in the sense that the image of
\begin{align} \label{eisintmap}
    \Eis \colon \cohomt^1(\partial X_\Gamma;{\Ocal_F}) \longrightarrow H^1(Y_\Gamma; F)
\end{align}
is not contained in $\cohomt^1(Y_\Gamma;{\Ocal_F})$. Hence we have have two integral structures on $H_{\Eis}^1(Y_\Gamma;\C)$. The first one is given by the integral Eisenstein classes
\begin{align}
    \Lcal_{0}\coloneqq \cohomt^1(X_\Gamma;{\Ocal_F}) \cap H_{\Eis}^1(X_\Gamma;\C) \subset H_{\Eis}^1(X_\Gamma;\C).
\end{align} The second one is coming from the integral structure on the boundary. We denote by 
\begin{align}
    \Lcal_{\Eis}\coloneqq \Eis(\cohomt^1(\partial X_\Gamma;{\Ocal_F}))
\end{align}
the image of the integral cohomology by $\Eis$. We call the denominator of the Eisenstein cohomology the $\Ocal_F$-ideal
\begin{align}
    \Den(\Lcal_{\Eis}) \coloneqq \left \{ \left . \lambda \in \Ocal_F \right \vert \lambda \Lcal_{\Eis} \subset \Lcal_0 \right \},
\end{align}
that relates these two natural integral structures. It was previously studied for Hilbert modular varieties in \cite{maennel}, for the degree 2 cohomology on Bianchi manifolds in \cite{feldhusen} and in the special case $K=\Q(i)$ in \cite{koenig}. See also \cite{harder2} for the latter case. Finally, in the case of Bianchi manifolds, Berger \cite{berger1} proves a lower bound on the denominator ideal in terms of an $L$-function. The main result of this paper is to give an upper bound.

By composing the Hecke character $\chi$ with the norm from $H$ to $K$ we get a Hecke character $\chi \circ \N_{H/K}$ of the same type on the Hilbert class field. Let $L(\chi \circ \N_{H/K},0)$ be the associated Hecke $L$-function at $s=0$. Let $L^{\integ}(\chi \circ \N_{H/K},0)$ be the normalization by a suitable and canonical complex period to make it an algebraic integer. 

\begin{thm*}(Theorem \ref{mainthm}) \label{thmdenomintro} We have the upper bound (in the sense of divisibility) on the denominator
\begin{align}\frac{1}{2\sqrt{D}}L^{\integ}(\chi \circ \N_{H/K},0) {\Ocal_F} \subset \Den(\Lcal_{\Eis}).
\end{align}
\end{thm*}
\noindent The theorem tells us that given a class $\Eis(\omega)$ in $\Lcal_{\Eis}$, we need to multiply it at most by the value $2^{-1}D^{-\frac{1}{2}}L^{\integ}(\chi \circ \N_{H/K},0)$ to make it integral.

We can then combine this result with Berger's result to obtain an equality in certain cases. Suppose that $K=\Q(\sqrt{D})$ has class number one and no non-trivial units. Then the Hilbert class group is $H=K$. Furthermore, since the image of $\chi$ will lie in $K$, we can take $F=K$ as well. Let $\q$ be a prime ideal of $\Ocal$. Let $K_\q$ be the completion at $\q$ with ring of integers $\Ocal_{\q}$. We fix a complex embedding of $K_\q$ and consider cohomology with coefficients in $\Ocal_\q$. We define the lattice
\begin{align} \label{eisatq}
    \Lcal_{\Eis,\q}\coloneqq \Eis(\cohomt^1(\partial X_\Gamma;{\Ocal_\q}))
\end{align}
and the denominator at $\q$
\begin{align}
    \Den(\Lcal_{\Eis,\q}) \coloneqq \left \{ \left . \lambda \in \Ocal_{\q} \right \vert \lambda  \Lcal_{\Eis,\q} \subset \cohomt_{\Eis}^1( X_\Gamma;{\Ocal_\q}) \right \}.
\end{align}
When $K$ has class number one, the algebraic $L$-function is $L^{\alg}(\chi,0)=\frac{1}{2}G_2(L_\Ocal)$ where 
\begin{align}
    G_2(L) \coloneqq \restr{\sideset{}{'}\sum_{\omega \in L} \frac{1}{\omega^2 \vert \omega \vert^\lambda}}{\lambda=0}
\end{align}
and $L_\Ocal=\Omega^{-2}\Ocal$ is a suitable normalization of the lattice $\Ocal \subset K$. The values of $G_2(L_\Ocal)$ are well known and are integral away from $2$, see Table \ref{tab:table1}. After comparing our setting to Berger's setting and combining both results we obtain the following.
\begin{cor*}(Corollary \ref{maincor}) Let $\q$ be a prime ideal coprime to $2D$ and suppose that $K$ has class number one and no non-trivial units. Then the denominator at $\q$ of the Eisenstein cohomology is exactly
    \[ \Den(\Lcal_{\Eis,\q})=G_2(L_\Ocal)\Ocal_{\q}. \]
\end{cor*}
\begingroup
\begin{table}[h!] \captionsetup{width=.6\textwidth,font={normalsize,it},justification=centering}
\renewcommand{\arraystretch}{3}
    \centering
    \begin{tabular}{ |c|c|c|c|c|c|c|c| } 
  \hline
  $D$ &  $-7$ & $-8$ & $-11$ & $-19$ & $-43$ & $-67$ & $-167$ \\
  \hline
  $G_2(L_\Ocal)$ & $\frac{1}{2}$ & $\frac{1}{2}$ & $2$ & $2$ & $12$ & $2\cdot 19$ & $2^2 \cdot 181$ \\
  \hline
\end{tabular}
    \caption{The values of $G_2(L_\Ocal)$ for the seven imaginary quadratic fields of class number one and having no non-trivial units.}
    \label{tab:table1}
\end{table}
\endgroup
\begin{rmk} Berger's result gives a lower bound on the denominator of the Eisenstein cohomology $\Lcal_{\Eis}(S_{K_\fin})$ of the adelic space $S_{K_\fin}$ associated to $\SL_2(K)$. The space $S_{K_\fin}$ has several connected components, and one of them is $Y_\Gamma$. Hence, since $Y_\Gamma \subset S_{K_\fin}$ we have $\Den(\Lcal_{\Eis}(S_{K_\fin})) \subset \Den(\Lcal_{\Eis})$. Thus, in general we cannot combine Berger's lower bound on $\Den(\Lcal_{\Eis}(S_{K_\fin}))$ with our upper bound on $\Den(\Lcal_{\Eis})$. However, when the class number is one we have $S_{K_\fin}=Y_\Gamma$, yielding the equality of the corollary.
In the case where the class number of $K$ is greater than one, one would need to give a lower bound on the denominator of the Eisenstein cohomology of the other connected components of $S_{K_\fin}$; see Remark \ref{indications}.
\end{rmk}

\paragraph{A few words on the proof of the theorem.} Let $\chi$ be an unramified Hecke character of type $(-2,0)$. For every ideal class $\afrak$ corresponding to a cusp of $Y_\Gamma$, we have a form $\omega_{\chi,\afrak}$ on the corresponding boundary component. They are integral forms spanning the cohomology of the boundary $H^1(\partial X_\Gamma;\C)$. Their image by the Eisenstein map are forms $E_{\chi,\afrak}$ that span the $h$-dimensional Eisenstein cohomology $H_{\Eis}^1(X_\Gamma;\C)$. These forms are not integral, they define classes in $H_{\Eis}^1(X_\Gamma;F)$ but not in $\cohomt_{\Eis}^1(X_\Gamma;\Ocal_F)$.

On the other hand, we have another basis $\Ehat_{\chi,\afrak}$ of the Eisenstein cohomology. These forms appear in the work of Ito \cite{ito}. A more general construction of these Eisenstein classes
 was done by Bergeron-Charollois-Garcia in \cite{bcg,bcg21} using the Mathai-Quillen formalism. 
 
 For a fractional ideal $\afrak$ of $K$ we define the Sczech cocycle $\Phi_{\afrak} \colon \Gamma \longrightarrow \C$ by
\begin{align}
    \Phi_\afrak \begin{pmatrix}
    a & b \\ c & d
    \end{pmatrix} \coloneqq 
    \begin{cases}
    I \left ( \frac{a+d}{c} \right )G_2(\afrak)-D(a,c,\afrak) & \textrm{if} \quad c \neq 0, \nonumber \\[0.3cm]
   I \left ( \frac{b}{d} \right )G_2(\afrak) & \textrm{if} \quad c = 0,
    \end{cases}
\end{align}
where $G_2(\afrak)$ is an Eisenstein series and $D(a,c,\afrak)$ a Dedekind sum; see Section \ref{eisdefs} for the definitions. The form $\Ehat_{\chi,\afrak}$ is related to the Sczech cocycle as follows
\begin{align*}  \Phi_\afrak(\gamma)=\chi(\afrak)\int_{u_0}^{\gamma u_0} \Ehat_{\chi,\afrak},
\end{align*}
 where $u_0$ is any point on $\HH_3$ and $\gamma \in \Gamma$. This formula is proved by using the idea of \cite{bcg21} to move the path of integration $[u_0 , \gamma u_0]$ to infinity. More precisely, choose a cusp $r$ of $Y_\Gamma$ and let $[r,\gamma^{-1} r]$ be the modular symbol joining the two cusps $r$ and $\gamma^{-1} r$. There is a homotopy between $[u_0 , \gamma u_0]$ and $[r,\gamma^{-1} r]$; see Figure \ref{movingtoboundarya}. The integral along the modular symbol $[r,\gamma^{-1} r]$ gives the Dedekind sum, whereas the term $I \left ( \frac{a+d}{c} \right )G_2(\afrak)$ is a contribution from the cusps. Note that this formula already appears in the work of Ito \cite[Theorem.~3]{ito}. A similar relation between Harder's Eisenstein cohomology for adelic spaces and Sczech's cocycle appears in work of Weselmann, see in particular \cite[Bemerkung.~2 on p.~116]{weselmann}.

 After a suitable normalization, the Sczech cocycle takes values in $\Ocal_H$. Hence, contrary to the forms $E_{\chi,\afrak}$, the forms  $\Ehat_{\chi,\afrak}$ are integral and define classes in $\cohomt^1(X_\Gamma;{\Ocal_F})$.
The two bases $E_{\chi,\afrak}$ and $\Ehat_{\chi,\afrak}$ of the Eisenstein cohomology $H^1(X_\Gamma;\C)$ are related by a matrix $M_\chi$ in $\Mat_h(F)$ whose determinant is $L(\chi \circ \N_{H/K},0)$. This explains the appearance of the $L$-function in the denominator.

\subsection*{Acknowledgements} This project is a generalization of some results of my thesis. I thank my advisors Nicolas Bergeron and Luis Garcia for suggesting me this topic and for their support. I was funded from the European Union’s Horizon 2020 research and innovation programme under the Marie Skłodowska-Curie grant agreement N$\textsuperscript{\underline{\scriptsize o}}$754362 \includegraphics[width=6mm]{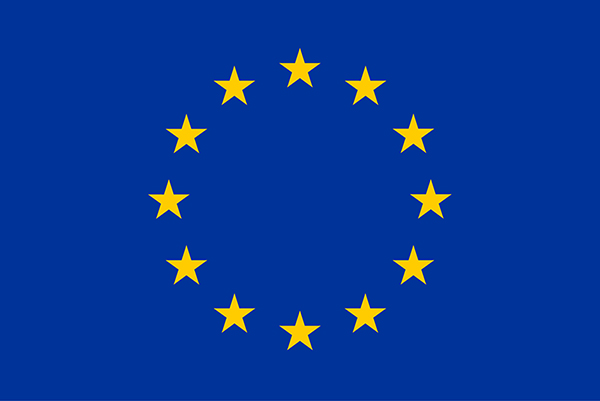} during the first stages of this project, then by the University of Manitoba. I thank Pierre Charollois and Sheng-Chi Shih for helpful comments that gretaly improved the exposition.

\section{Eisenstein series, $L$-functions and Sczech's cocycle} \label{eisdefs}

Let $K=\Q(\sqrt{D})$ be a quadratic imaginary field with no non-trivial units. Suppose that $K$ has class number one or that $D \equiv 1 \pmod 8$. Let $\Ocal$ be its ring of integers, that we view as a lattice in $\C$ after fixing some embedding of $K$ in $\C$. 

\subsection{Canonical periods} \label{seccanperiod}

\paragraph{When the class number of $K$ is larger than one.} We follow \cite[Section.~5, p.~103]{sczech2}. Let $\tau=\frac{1+\sqrt{D}}{2}$ so that $\Ocal=\Z+\tau\Z$. Let
\begin{align}
    \eta(\tau)=e^{\frac{\pi i \tau}{12}} \prod_{n=1}^{\8} \left (1-e^{2 n i \pi \tau} \right )
\end{align}
be the Dedekind eta function and define 
\begin{align}
    u \coloneqq -\frac{2^{12}\eta(2\tau)^{24}}{\eta(\tau)^{24}}.
\end{align}
It is a unit in the Hilbert class field $H$. Define the elliptic curve
\begin{align}
   E : y^2=4x^3-ax-b
\end{align}
where 
\begin{align}
a \coloneqq 12D(u-16), \qquad b \coloneqq \left (2\sqrt{D} \right )^3 \sqrt{u(j-1728)},   
\end{align}
$j$ is the $j$-invariant $j(\tau)$ and the square root is chosen such that $b$ is a positive real number. Then $a$ and $b$ are in $\Ocal_H$, the period lattice is $L_\Ocal=\Omega \Ocal$ where
\begin{align} \label{perioddef}
    \Omega \coloneqq \frac{\pi}{\left ( 144 \vert D \vert \right )^{\frac{1}{4}}}\frac{\eta(\tau)^4}{\eta(2\tau)^2},
\end{align}
and the discriminant of the elliptic curve is 
\begin{align}
    \Delta(L_\Ocal)=12^6D^3u.
\end{align}
The Weierstrass equation is related to the period lattice $L_\Ocal$ by $a=g_2(L_\Ocal)$ and $b=g_3(L_\Ocal)$, where for any lattice $L\subset \C$ we define
\begin{align}
    g_2(L) \coloneqq 60\sideset{}{'} \sum_{\omega \in L} \frac{1}{\omega^4}, \qquad g_3(L) \coloneqq 140\sideset{}{'} \sum_{\omega \in L} \frac{1}{\omega^6}.
\end{align}

For any fractional ideal $\afrak$ of $K$ let $\sigma=\sigma_\afrak=(\afrak,H/K)$ be the Artin symbol. Let $E_{\afrak} =E^{\sigma_\afrak}$ be the elliptic curve given by the Weierstrass equation
\begin{align} \label{conjell}
    E_{\afrak} \colon y^2=4x^3-a^{\sigma}x-b^{\sigma}, 
\end{align} 
which is also defined over $\Ocal_H$. Since $j(\Ocal)^{\sigma_\afrak}=j(\afrak^{-1})$ we have that $E_{\afrak} \simeq \C/L_{\afrak}$ where $L_{\afrak} = \Omega(\afrak)\afrak^{-1}$ for some complex period $\Omega(\afrak) \in \C^\times$ that we fix. Hence $g_k(L_\afrak)=g_k(L)^{\sigma}$ is is in $\Ocal_H$ for $k=2,3$. Let $\lambda(\afrak)$ be the complex number
\begin{align}
    \lambda(\afrak) \coloneqq \frac{\Omega(\afrak)}{\Omega},
\end{align}
so that $L_\afrak=\lambda(\afrak)\afrak^{-1}L_\Ocal$.
It has the following properties \cite[Appendix.~D(e) and D(f) on p.~371]{robert}
\begin{enumerate}
\item $\lambda(\afrak) \in H^\times$,
\item $\lambda(\alpha\afrak)=\alpha\lambda(\afrak)$,
\item $\lambda(\afrak\cfrak)=\lambda(\cfrak)\lambda(\afrak)^{\sigma_\cfrak}$.
\end{enumerate}
\paragraph{When the class number of $K$ is one.} Suppose  that $K=\Q(\sqrt{D})$ has class number one and has non-trivial units. This is the case for the seven values of $D$ listed in the table below and extracted from \cite[Tableau.~B.1]{robert}. None of these values of $D$ is congruent to $1$ modulo $8$, so we cannot use the above construction of $\Omega$.
For each value of $D$ consider the elliptic curve
\begin{align}
   E : y^2=4x^3-ax-b
\end{align}
where $a$ and $b$ are liste below in Table \eqref{tableab}. Comparing the invariant $j(E)=1728\frac{a^3}{a^3-27b^2}$ with a list\footnote{See for example \cite[p.261]{cox}.} of $j$-invariants $j(\Ocal)$ we find that they agree. Hence, the elliptic curve $E$ has complex multiplication by $\Ocal$ and $E(\C) \simeq \C / L_\Ocal$ where $L_\Ocal = \Omega \Ocal$ for some complex period $\Omega \in \C^\times$ that we fix.

\begingroup
\begin{table}[h!] \captionsetup{width=.75\textwidth,font={normalsize,it},justification=centering}
\renewcommand{\arraystretch}{2}
    \centering
    \begin{tabular}{ |c|c|c|c|c|c|c|c|c|c| } 
  \hline
  $D$  & $-7$ & $-8$ & $-11$ & $-19$ & $-43$ & $-67$ & $-167$ \\
  \hline
  $a$ & $5\cdot7 $ & $2 \cdot 3 \cdot 5 $ & $2^3 \cdot 3 \cdot 11$ & $2^3 \cdot 19$ & $2^4 \cdot 5 \cdot 43$ & $2^3 \cdot 5 \cdot 11 \cdot 67$ & $2^4 \cdot 5 \cdot 23 \cdot 29 \cdot163$ \\
  \hline
  $b$ & $7^2$ & $2^2 \cdot7$ & $7 \cdot11^2$ & $19^2$ & $3 \cdot 7 \cdot 43^2$ & $7 \cdot 31 \cdot 67^2$ & $7 \cdot 11 \cdot 19 \cdot 127 \cdot 163^2$ \\
  \hline
\end{tabular}
    \caption{The coefficients $a$ and $b$ for the seven imaginary quadratic fields of class number one and having no non-trivial units.}
    \label{tableab}
\end{table}
\endgroup

\subsection{Kronecker-Eisenstein series} Let $L=\Z\omega_1+\Z \omega_2$ be any lattice in $\C$ where $\omega_1$ and $\omega_2$ are complex numbers such that  $\omega_1/\omega_2$ has positive imaginary part. We define
\begin{align}
\Dcal(L) \coloneqq I(\omega_1\overline{\omega}_2)=2 i \vert \omega_2 \vert^2 \im \left ( \frac{\omega_1}{\omega_2}\right ).
\end{align}
where $I(z)\coloneqq z-\bar{z}$. The order of $L$ is the order
\begin{align}
    \Ocal(L) \coloneqq \left \{ \left . \lambda \in \C \right \vert  \lambda L=L \right \}
\end{align}
of $\Ocal$, and the lattice $L$ is homothetic to an ideal in $\Ocal(L)$. In particular, if $L=\afrak$ is a fractional ideal then $\Ocal(\afrak)=\Ocal$.
Consider the character
\begin{align}
    \theta(z) \coloneqq \exp \left (2i\pi{\frac{I(z)}{\Dcal(L)}} \right ).
\end{align}
For a non-negative integer $k$ and a complex number $s$ we define the Kronecker-Eisenstein series

\begin{align}
    G(s,k,p,q,L) \coloneqq \sideset{}{'}\sum_{\omega \in L} \theta(w\bar{p})\frac{\overline{q+\omega}^k}{\vert q+\omega \vert^{2s+k}},
\end{align}
which converges for $\re(s)>1$ and the $'$ means that we remove $\omega=-q$ from the summation if $q$ is in $L$. This is the series considered by Weil in \cite[section~VIII]{weilelliptic}. The function admits a meromorphic continuation to the whole plane with only possible poles at $s=0$ (if $k=0$ and $q$ is in $L$) and at $s=\frac{1}{2}$ (if $k=0$ and $p$ is in $L$); see \cite[section~VIII, p.~80]{weilelliptic}. Moreover, it satisfies the functional equation
\begin{align} \label{funceq2}
    \Ecal(s,k,p,q,L)=\theta(p\bar{q})\Ecal(1-s,k,p,q,L)
\end{align}
where
\begin{align}
    \Ecal(s,k,p,q,L) \coloneqq \left ( \frac{2i\pi}{\Dcal(L)}\right )^{-s}\Gamma \left (s+\frac{k}{2} \right )G(s,k,p,q,L).
\end{align}
For a positive integer $k$ we set
\begin{align}
G_k(z,L) \coloneqq G\left ( \frac{k}{2},k,0,z,L\right )= \restr{\sideset{}{'}\sum_{\omega \in L} \frac{1}{(z+\omega)^k \vert z+k \vert^\lambda}}{\lambda=0}
\end{align}
and
\begin{align}
G(z,L) \coloneqq \frac{2i\pi}{\Dcal(L)}G(0,2,0,z,L).
\end{align}
They satisfy the following homogeneity properties
\begin{align} \label{propeis}
    G_k(\alpha z, \alpha L) & =\alpha^{-k}G_k(z,L) \nonumber \\
    G(\alpha z, \alpha L) & = \frac{\bar{\alpha}}{\alpha} G(z,L).
\end{align}
When $z=0$ we set
\begin{align}
    G(L) & \coloneqq G(0,L), \nonumber \\
    G_k(L) & \coloneqq G_k(0,L).
\end{align}

Let $L_\afrak$ be the lattices of Section \ref{seccanperiod}. The following result is well known and was first proved by Damerell \cite{damerell}. See also \cite[Appendix D(c)]{robert}.
\begin{prop} \label{eisint}
    The value $2\sqrt{D}G_2(L_\afrak)$ is an algebraic integer in $\Ocal_H$.
\end{prop}

\subsection{Sczech cocycle}
Let $L \subset \C$ be a lattice such that $\Ocal(L)=\Ocal$. For $a$ and $c$ in $\Ocal$, and $c$ nonzero, we define the Dedekind sum
\begin{align}
    D(a,c,L) \coloneqq \frac{1}{c} \sum_{r \in L / c L} G_1 \left ( \frac{ar}{c},L \right )G_1 \left ( \frac{r}{c},L \right ).
\end{align}
In \cite{sczech}, Sczech shows that the map $\Phi_L \colon \Gamma \longrightarrow \C$ defined by
\begin{align}
    \Phi_L \begin{pmatrix}
    a & b \\ c & d
    \end{pmatrix} = 
    \begin{cases}
    I \left ( \frac{a+d}{c} \right )G_2(L)-D(a,c,L) & \textrm{if} \quad c \neq 0, \nonumber \\
   I \left ( \frac{b}{d} \right )G_2(L) & \textrm{if} \quad c = 0
    \end{cases}
\end{align}
is a cocycle.

\begin{rmk}   
We will show in Theorem \ref{eisensteinszcech} that $\Phi_{\afrak}(\gamma)=\chi(\afrak)\int_{u_0}^{\gamma u_0}\Ehat_{\chi,\afrak}$ for a certain closed form $\Ehat_{\chi,\afrak}$ in $\Omega^1(Y_\Gamma)$. Since the form is closed, the integral is independent of the basepoint $u_0$ and we have
\begin{align}
    \Phi_{\afrak}(\gamma_1\gamma_2) & = \int_{u_0}^{\gamma_1 \gamma_2 u_0}\Ehat_{\chi,\afrak} \nonumber \\
    & =\int_{\gamma_2 u_0}^{\gamma_1 \gamma_2 u_0}\Ehat_{\chi,\afrak}+\int_{u_0}^{\gamma_2 u_0}\Ehat_{\chi,\afrak} \nonumber \\
   & =\int_{u_0}^{\gamma_1 u_0}\Ehat_{\chi,\afrak}+\int_{u_0}^{\gamma_2 u_0}\Ehat_{\chi,\afrak} \nonumber \\
   & = \Phi_{\afrak}(\gamma_1) + \Phi_{\afrak}(\gamma_2).
\end{align}
This gives an alternative proof of the cocycle property of the Sczech cocycle $\Phi_\afrak$.
\end{rmk}

Furthermore, Sczech proved that the cocycle is integral. The following is proved in \cite[Satz.~4]{sczech}.

\begin{prop}[Sczech] \label{sczechintegral} For any $\gamma$ in $\Gamma$, the value $2\Phi_{L_\afrak}(\gamma)$ is an algebraic integer in $\Ocal_H$.
\end{prop}

\subsection{Hecke characters and $L$-functions}
Let $\chi$ be an unramified algebraic Hecke character of infinity type $(k,j)$. Let $F_\chi$ be the finite extension of $H$ obtained by adjoining the image of $\chi$. So $\chi$ is a character
\begin{align}
    \chi \colon \Ical_K \longrightarrow F_\chi^\times \subset \C^\times
\end{align}
on the group $\Ical_K$ of fractional ideals of $K$ such that on principal ideals $(\alpha)$ in $K$ we have 
\begin{align}
    \chi((\alpha))=\alpha^{k}\bar{\alpha}^j.
\end{align}

Let $F$ be a Galois extension of $H$ containing $F_\chi$ and the $h$-th roots of unity. If $\chi$ and $\widetilde{\chi}$ are two Hecke characters of the same type, then there is a character $\varphi \in \widehat{\Cl(K)}$ on the class group such that $\widetilde{\chi}=\varphi \chi$. Thus $F$ is a field containing all the fields $F_\chi$ as $\chi$ varies over all the Hecke characters as above.

\begin{prop} \label{proplambda}
For any fractional ideal $\afrak$ and any Hecke character $\chi$ of infinity type $(-2,0)$, the value $\chi(\afrak)\lambda(\afrak)^2$ is a unit in $\Ocal_{F}$.
\end{prop}
\begin{proof}
 For $\sigma \in \Gal(H/K)$ we have $\Delta(L_\Ocal)^\sigma=12^6D^3u^{\sigma}$ where $u^{\sigma} \in \Ocal_H^\times$ is a unit. Hence
\begin{align}
c(\sigma) \coloneqq \frac{\Delta(L_\Ocal)}{\Delta(L_\Ocal)^\sigma}=\frac{u}{u^\sigma}   
\end{align}
is a unit. When $\sigma=\sigma_\afrak$ then $c(\sigma_\afrak)=\Delta(L_\Ocal)/\Delta(L_\afrak)$ . 
 Since $\Delta(\alpha L_\Ocal)=\alpha^{-12}\Delta(L_\Ocal)$, we have
    \begin{align} \label{lambdaeq}
       c(\sigma_\afrak)= \Delta(L_\Ocal)/\Delta(L_\afrak) =\frac{\Omega^{-12}\Delta(\Ocal)}{\Omega(\afrak)^{-12}\Delta(\afrak^{-1})}=\lambda(\afrak)^{12} \frac{\Delta(\Ocal)}{\Delta(\afrak^{-1})}.
    \end{align}
Moreover, by \cite[p.49]{deshalit} or\footnote{Note that in \cite{langdioph} the result is proved for a split prime ideal $\p$. However, any class in $\Cl(K)$ contains such an ideal and equality \eqref{lambdaeq} only depends on the ideal class of $\afrak$.} \cite[Theorem.~5 p.165]{langdioph} we have 
    \begin{align}
        \frac{\Delta(\afrak^{-1})}{\Delta(\Ocal)} \Ocal_H=\afrak^{12} \Ocal_H.
    \end{align}
    Combining the two gives $\lambda(\afrak)^{12}\Ocal_H=\afrak^{12} \Ocal_H$, and by comparing the prime factorizations we get 
\begin{align} \label{idealequality1}
    \lambda(\afrak) \Ocal_H=\afrak\Ocal_H.
\end{align}
    
On the other hand, since $\afrak^h=\alpha \Ocal$ for some $\alpha \in K^\times$ and $h$ is the class number, we have
$\chi(\afrak)^h\Ocal_{F}=\afrak^{-2h}\Ocal_{F}$. By comparing the prime decomposition we then also get 
\begin{align} \label{idealequality2}
\chi(\afrak)\Ocal_{F}=\afrak^{-2}\Ocal_{F}. 
\end{align}
Combining \eqref{idealequality1} and \eqref{idealequality2} we obtain
\begin{align}
  \chi(\afrak)\lambda(\afrak)^2\Ocal_{F}=\Ocal_{F}.  
\end{align}
\end{proof}
Let $\chi$ be a Hecke character of infinity type $(k,j)$. The Hecke $L$-function
\begin{align}
    L(\chi,s) \coloneqq \sum_{0 \neq \afrak \subseteq \Ocal} \frac{\chi(\afrak)}{\N(\afrak)^{s}}
\end{align}
converges for $\re(s)>1+\frac{k+j}{2}$, admits a meromorphic continuation to the whole plane and a functional equation.

Let $\{\afrak_1,\dots,\afrak_h\}$ be integral ideal representatives of the class group of $K$. Every integral ideal $\afrak$ can be written $\afrak=(\alpha)\afrak_i$ for some $i \in \{1, \dots,h \}$, where $(\alpha) \subset \afrak_i^{-1}$. Hence we can write the $L$-function as
\begin{align}
    L(\chi,s) & = \sum_{i=1}^h \frac{\chi(\afrak_i)}{\N(\afrak_i)^{s}} \frac{1}{w(\afrak_i^{-1})}\sum_{\alpha \in \afrak_i^{-1}-\{0\}} \frac{\chi((\alpha))}{(\alpha \bar{\alpha})^s} \nonumber \\
    & = \sum_{i=1}^h \frac{\chi(\afrak_i)}{\N(\afrak_i)^{s}} \frac{1}{w(\afrak_i^{-1})} \sum_{\alpha \in \afrak^{-1}_i-\{0\}} \frac{\alpha^{k}\bar{\alpha}^j}{\vert \alpha \vert^{2s}}, \nonumber \\
    & = \sum_{i=1}^h \frac{\chi(\afrak_i)}{\N(\afrak_i)^{s}} \frac{1}{w(\afrak_i^{-1})} G\left (s-\frac{k+j}{2},j-k,0,0;\afrak^{-1}_i \right )
\end{align}
where $w(\afrak_i)=2^{\vert \Ocal^\times \cap \afrak_i \vert } \in \{1,2\}$ depends on the number of units in $\afrak_i$. 

From now on, let $\chi$ be of infinity type $(-2,0)$. At $s=0$ we then have
\begin{align} \label{Lfuncsplit}
    L(\chi,0)=\sum_{i=1}^h \frac{\chi(\afrak_i)}{w(\afrak_i)}G_2(\afrak^{-1}_i).
\end{align}
Note that \eqref{Lfuncsplit} does not depend on the choice of representatives.
We define the algebraic $L$-function
\begin{align} \label{Lalg}
    L^{\alg}(\chi,s) \coloneqq \Omega^{-2}L(\chi,s),
\end{align}
where $\Omega$ is our the canonical period \eqref{perioddef}. We also define the integral $L$-function
\begin{align}
    L^{\integ}(\chi,s)\coloneqq 4\sqrt{D}L^{\alg}(\chi,s).
\end{align}
The normalizations are chosen such that we have the following result.
\begin{prop} \label{Lalgint} The value $L^{\alg}(\chi,0)$ is an algebraic number in $F$ and nonzero. More precisely, the value $L^{\integ}(\chi,0)$ is in $\Ocal_F$.
\end{prop}
\begin{proof} The integrality statement is due Damerell and follows from Proposition \ref{eisint}. By the homogeneity of $G_2$ we have $G_2(\afrak_i^{-1})=\Omega(\afrak_i)^2G_2(L_{\afrak_i})$ and hence
\begin{align}
    4 \sqrt{D}\Omega^{-2}L(\chi,0)=\sum_{i=1}^h \chi(\afrak_i)\lambda(\afrak_i)^2\frac{2}{w(\afrak_i^{-1})}2\sqrt{D}G_2(L_{\afrak_i}).
\end{align}
and all the terms are algebraic. 
Furthermore, by Proposition \ref{proplambda} we have that $\chi(\afrak_i)\lambda(\afrak_i)^2$ is a unit in $\Ocal_F$. The integrality of $L^{\integ}(\chi,0)$ follows since $2\sqrt{D}G_2(L_{\afrak_i})$ is in $\Ocal_H$ by Proposition \ref{eisint}, and $w(\afrak_i^{-1})=1$ or $2$.

The character $\chi$ has infinity type $(-2,0)$. Hence the Hecke character $\chi(\afrak)\N(\afrak)^2$ has type $(0,2)$, and $\Ccal(\afrak) \coloneqq \overline{\chi(\afrak)\N(\afrak)^2}$ has type $(2,0)$. Then we have
\begin{align}
    L(\chi,0)=L(\chi(\afrak)\N(\afrak)^2,2)= \frac{2\pi}{\sqrt{\vert D \vert}}L(\Ccal,1),
\end{align}
where the second equality is the functional equation; see \cite{Lgreenberg}. By arguing as in \cite{Lgreenberg} (in the case $n=2$), we have that $L(\Ccal,1)$ is nonzero.
\end{proof}

Composing with the norm, one gets a Hecke character on $H$
\begin{align}
    \chi \circ \N_{H/K} \colon \Ical_H \longrightarrow F^\times \subset \C^\times
\end{align}
of infinity type $(-2,0)$, where for any fractional ideal $\bfrak$ of $H$
\begin{align}
 \N_{H/K}(\bfrak)= K \cap \prod_{\sigma \in \Gal(H/K)} \sigma(\bfrak)  
\end{align}
is the relative ideal norm. We can then also consider the $L$-function
\begin{align}
 L(\chi \circ \N_{H/K},s)= \sum_{0 \neq \bfrak \subseteq \Ocal_H} \frac{\chi(\N_{H/K}(\bfrak))}{\N(\bfrak)^{s}}=\prod_{\varphi \in \widehat{\Cl(K)}}L(\varphi\chi,s).
\end{align}
It follows from Proposition \ref{Lalgint} that
\begin{align}
   L^{\alg}(\chi \circ \N_{H/K},0) \coloneqq \Omega^{-2h}L(\chi \circ \N_{H/K},0)
\end{align}
is a nonzero algebraic number in $F$ and that
\begin{align}
   L^{\integ}(\chi \circ \N_{H/K},0) \coloneqq 4^hD^{\frac{h}{2}}L^{\alg}(\chi \circ \N_{H/K},0)
\end{align}
is an algebraic integer in $\Ocal_F$.

\begin{rmk}
   Proposition \ref{Lalgint} on the integrality of $L^{\integ}(\chi,0)$ for a character of type  $(k,j)=(-2,0)$ is due to Damerell \cite{damerell}. Other integrality results are also know for more general characters of type $(k,j)$ with $k<0$ and $j \geq 0$ (so called {\em critical type}). We refer to \cite[Theorem.~3]{berger_2009} for a summary of $p$-integrality results due to Shimura, Coates-Wiles, Katz, Hida, Tilouine, de Shalit, Rubin and others. More recently, Kings and Sprang show in \cite{KS} that for any pair of coprime ideals in $K$ we have
    \begin{align}
     (1-\chi(\cfrak'))(\chi(\cfrak)\N(\cfrak)-1) L^{\alg}(\chi,0) \in \Ocal_{F_\chi} \left [ \frac{1}{d_K\N(\cfrak\cfrak')}\right ]
    \end{align}
    where
    \begin{align}
      L^{\alg}(\chi,0) \coloneqq \frac{(-k-1)!\Omega^k(2i\pi)^j}{(\Omega^\vee)^j}  L(\chi,0) 
    \end{align}
    for a suitable choice of complex periods $\Omega$ and $\Omega^\vee$ and $d_K$ is the discriminant of $K$.
\end{rmk}

\section{Eisenstein cohomology} \label{eisensteincohom}

Let $\HH_3$ be the hyperbolic $3$-space
\begin{align}
\HH_3 \coloneqq \left \{ u=z+jv \mid z \in \C, v \in \R_{>0} \right \},
\end{align}
where $ij=-ji$ and $i^2=j^2=-1$. For $u=z+jv$ let $\overline{u}=\overline{z}-jv$ and $\vert u \vert = u \overline{u}= \vert z \vert ^2 + v^2$. The group $\SL_2(\C)$ acts transitively on $\HH_3$ by 
\begin{align}
u \longmapsto(au+b)(cu+d)^{-1}=\frac{(au+b)(\overline{cu+d})}{\vert cz+d\vert^2+v^2},
\end{align}
and the stabilizer of $j$ is $\SU(2)$. Hence the symmetric space $\SL_2(\C)/\SU(2)$ is isomorphic to $\HH_3$. For a fractional ideal $\bfrak$ of $K$ let
\begin{align}
    \Gamma(\bfrak) \coloneqq \left \{ \left . \begin{pmatrix}
    a & b \\ c & d
    \end{pmatrix} \in \SL_2(K) \right \vert a,d \in \Ocal, b \in \bfrak, c \in \bfrak^{-1} \right \}
\end{align}
be the subgroup of $\SL_2(K)$ preserving $\Ocal \oplus  \bfrak$ by right multiplication. Let $\Gamma \coloneqq \Gamma(\Ocal)=\SL_2(\Ocal)$ and
\begin{align}
 Y_\Gamma \coloneqq \Gamma \backslash \HH_3.
\end{align} 
This space is a non-compact $3$-dimensional orbifold and can be compactified in several ways, one of them is the Borel-Serre compactification.

\subsection{Borel-Serre compactification}

We describe the Borel-Serre compactification of $Y_\Gamma$; see \cite{borelji} or \cite{jimp} for more on compactifications of locally symmetric spaces. We define the space
\begin{align}
 \HH_3^{\ast}\coloneqq \HH_3 \cup \bigsqcup_{r \in \PP^1(K)} \Hcal_r
\end{align}
where $\Hcal_r=\PP^1(\C)-\{r\}$.  We have a canonical map
\begin{align} \label{Hrmap}
    \Hcal_r & \longrightarrow \C \nonumber \\
    (x:y) & \longmapsto \frac{mx}{my-nx}
\end{align}
where $r=(m : n)$. Hence we can view the space $\HH_3^\ast$ as adding a copy of $\C$ at every point $\frac{m}{n}$ on the boundary of $\HH_3$. The topology on $\HH_3^\ast$ is defined as follows: let $\Hcal_\8$ be the boundary component at $\8$ corresponding to $r=(1:0)$. A sequence $u_k =z_k+jv_k$ converges to $(z_0:1) \in \Hcal_\8$ if $\lim_{k \rightarrow \8} v_k=\8$ and $\lim_{k \rightarrow \8} z_k=z_0$. If $\gamma$ maps $\8$ to $r$ then $u_k$ converges to $(x:y) \in \Hcal_r$ if $\gamma^{-1}u_k$ converges to $\gamma^{-1}(x:y) \in \Hcal_\8$. The action of $\Gamma$ extends to $\HH_3^{\ast}$ by sending $z \in \Hcal_r$ to $\gamma z \in \Hcal_{\gamma r}$, where $\SL_2(\C)$ acts on $\PP^1(\C)$ by
\begin{align}
\begin{pmatrix} a & b \\ c & d  \end{pmatrix}(x : y)=(ax+by:cx+dy).  
\end{align}
We define $X_\Gamma \coloneqq \Gamma \backslash \HH_3^{\ast}$. Let $C_\Gamma  \coloneqq \Gamma \backslash \PP^1(K)$ be the set of cusps of $Y_\Gamma$. We can represent cusps by fractional ideals in $K$ since we have a bijection \cite[Theorem.~18]{berger_2009}
\begin{align} \label{clbij}
    C_\Gamma & \longrightarrow \Cl(K) \nonumber \\
    c=\Gamma (m:n) & \longmapsto [\afrak_c]
    \end{align}
where $\afrak_{c} \coloneqq (m)+(n)$ and $\Cl(K)$ is the class group. We have a bijection 
\begin{align}
    \Gamma \backslash \bigsqcup_{r \in \PP^1(K)} \Hcal_r & \longrightarrow \bigsqcup_{c=[r] \in C_\Gamma} \Gamma_r \backslash \Hcal_r \nonumber \\
    (x:y)  & \longmapsto \gamma (x:y)
\end{align}
where $(x:y) \in \Hcal_r$, the element $\gamma$ maps $(1:0)$ to some $[r] \in C_\Gamma$ and $\Gamma_r$ is the unipotent radical of the stabilizer $\left \{ \left . \gamma \in \Gamma \; \right \vert \; \gamma c=c \right \}$ of $r$ in $\Gamma$. 
The Borel-Serre compactification is 
\begin{align}
    X_\Gamma \coloneqq Y_\Gamma \cup \bigsqcup_{c=[r] \in C_\Gamma} \Gamma_r \backslash \Hcal_r,
\end{align}
and the boundary of the Borel-Serre compactification is 
\begin{align}
\partial X_\Gamma =  \bigsqcup_{c=[r] \in C_\Gamma} \Gamma_r \backslash \Hcal_r.   
\end{align}
Note that
\begin{align}
    \Gamma_\8 = \left \{ \left . \begin{pmatrix}
        1 & b \\ 0 & 1
    \end{pmatrix} \; \right \vert \; b\in \Ocal \right \} \simeq \Ocal,
\end{align}
hence
\begin{align} \label{amap}
    \Gamma_\8 \backslash \Hcal_\8 & \longrightarrow \C/\Ocal \nonumber \\
    \Gamma_\8 (z:1) & \longmapsto z+\Ocal.
\end{align}
Let $r=(m:n)$ and $M=\begin{pmatrix}
    y & -x \\ -n & m
\end{pmatrix}$ in $\SL_2(K)$ that maps $r$ to $\8$. We then have
\begin{align}
    M\Gamma_rM^{-1}=(M\Gamma M^{-1})_\8.
\end{align}
We define the fractional ideal $\afrak_M=(m)+(n)$, so that
\begin{align}
 M \in \begin{pmatrix}
\afrak_M^{-1} & \afrak_M^{-1} \\
\afrak_M & \afrak_M
 \end{pmatrix}   
\end{align}
and
\begin{align}
    M\Gamma M^{-1}=\Gamma(\afrak_M^{-2});
\end{align}
see for example \cite[p.~12]{geer}. In particular \[
     M\Gamma_rM^{-1} = (M\Gamma M^{-1})_\8 = \left \{ \left . \begin{pmatrix}
        1 & b \\ 0 & 1
    \end{pmatrix} \; \right \vert \; b\in \afrak_M^{-2} \right \}.
\]
Thus we have a map
\begin{align} \label{amap2}
    \Gamma_r \backslash \Hcal_r \longrightarrow (M\Gamma M^{-1})_\8 \backslash \Hcal_\8 \simeq \C/\afrak_M^{-2},
\end{align}
where the first map is $(z:1) \mapsto M(z:1)$ and the second is \eqref{amap}.

\subsection{Equivariant homology and cohomology}

Since $Y_\Gamma$ is not a manifold  but rather an orbifold, we need to work with equivariant (co)homology. 

\paragraph{Restriction of differential forms.} At $\8$ the boundary component of $\HH_3^\ast$ can be embedded inside $\HH_3$ as a horocycle by the map
\begin{align}\iota_v \colon \Hcal_\8\simeq \C & \hooklongrightarrow \HH_3 \nonumber \\ 
z & \longmapsto z+jv.
\end{align}
For a differential form $\omega$ in $\Omega^k(\HH_3)$ we define the restriction map
\begin{align}
\res_\8 \colon \Omega^k(\HH_3) & \longrightarrow \Omega^k(\Hcal_\8) \nonumber \\
    \omega & \longmapsto \res_{\8}(\omega) \coloneqq \lim_{v \rightarrow \8} i_v^\ast \omega
\end{align} to be the restriction of $\omega$ to the boundary component at the cusp $\8$. Note that $\iota_v$ is $\Ocal$-equivariant in the sense that for every $\gamma=\begin{pmatrix}
        1 & b \\ 0 & 1
    \end{pmatrix}$ in $\Gamma_\8$ we have 
\begin{align}
    \begin{pmatrix}
        1 & b \\ 0 & 1
    \end{pmatrix} \iota_v(z)=\iota_v(z+b).
\end{align}
Hence it descends to a map $\iota_v \colon \Gamma_\8 \backslash \Hcal_\8 \rightarrow Y_\Gamma$, and so does the restriction map:
\begin{align}
    \res_\8 \colon \Omega^k(Y_\Gamma) \longrightarrow \Omega^k(\Gamma_\8 \backslash \Hcal_\8).
\end{align}

The other boundary components $\Hcal_r$ can be embedded as horosphere at the other cusps. Let $M$ be a matrix in $\SL_2(K)$ sending $r$ to $\8$. We define the embedding
\begin{align}
    \iota_{v,M}(z) \coloneqq M^{-1} \circ \iota_v \circ M \colon \Hcal_r \longrightarrow \HH_3.
\end{align}
Its image is a horosphere, see figure \ref{horosphere}. If $N$ is another matrix sending $r$ to $\8$, then
\begin{align}
    M=\begin{pmatrix}
      a & b \\ 0 & a^{-1}  
    \end{pmatrix}N
\end{align}
and $\iota_{N,r}= M^{-1} \circ \iota_{v/a^2} \circ M$. Hence the embedding depends on the choice of $M$, but taking the limits gives a well-defined map
\begin{align}
\res_r \colon \Omega^k(\HH_3) & \longrightarrow \Omega^k(\Hcal_r) \nonumber \\
    \omega & \longmapsto \res_{r}(\omega) \coloneqq \lim_{v \rightarrow \8} i_{M,v}^\ast \omega,
\end{align}
which also descends to the quotient
\begin{align}
    \res_r \colon \Omega^k(Y_\Gamma) \longrightarrow \Omega^k(\Gamma_r \backslash \Hcal_r).
\end{align}
Finally, note that for any matrix $A$ in $\SL_2(K)$ and any form $\omega$ on $\HH_3$ we have
\begin{align} \label{equivariance}
    A^\ast \res_{Ar}(\omega)=\res_r(A^\ast \omega).
\end{align}
\begin{figure}[h]
\centering
\includegraphics[scale=0.22]{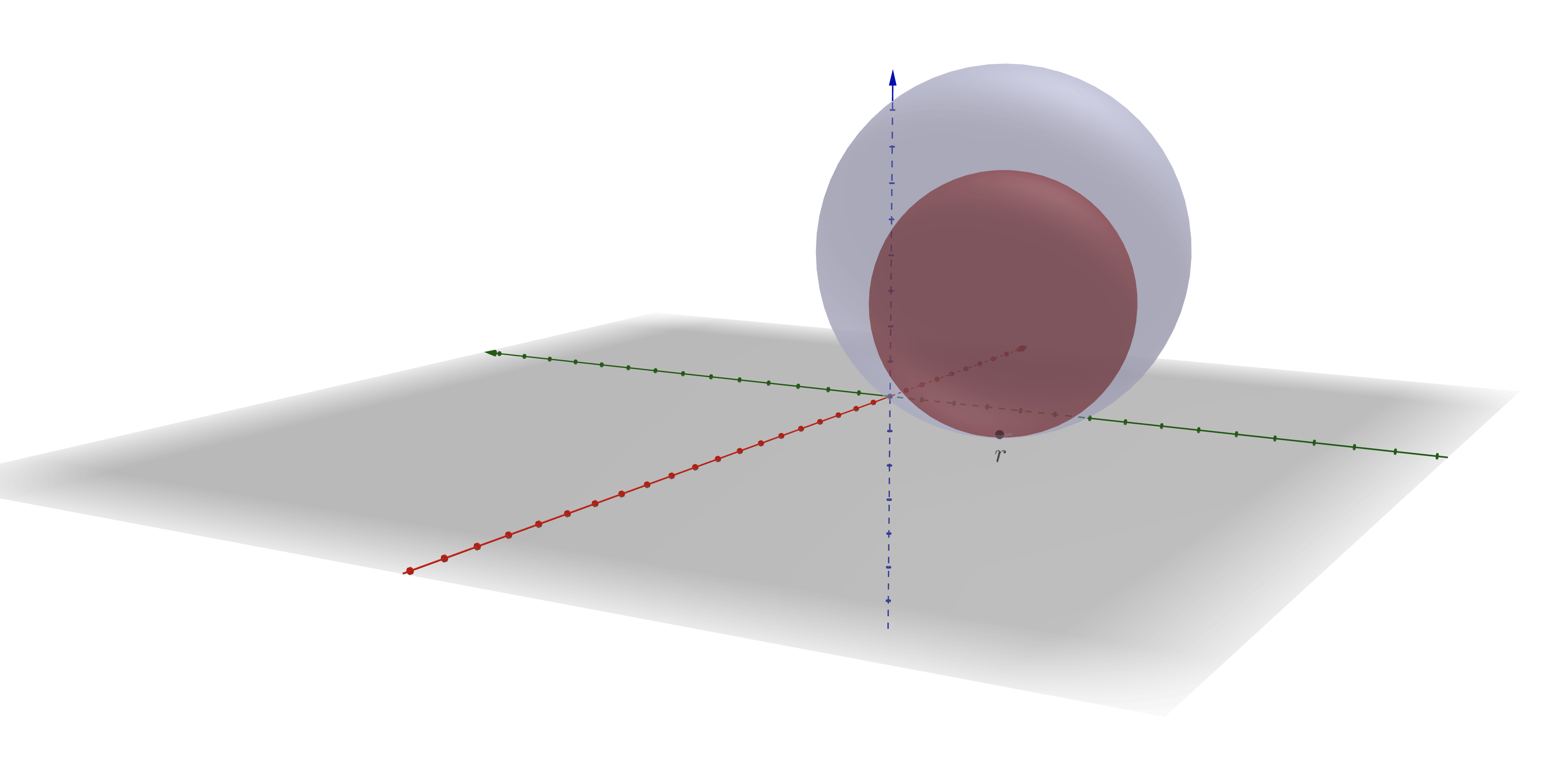}
\captionsetup{width=.75\textwidth,font={small,it},justification=centering}
\caption{The embedding of $\Hcal_r$ by $\iota_{r,v}$ is a horosphere in $\HH_3$, tangent to the plane $v=0$ at the cusp $r$. As $v$ increases, the radius of the sphere decreases. Hence we can see the boundary components as horospheres at infinity.}
\label{horosphere}
\end{figure}

\paragraph{Cohomology.} Following \cite[Section.~2]{stevens} we define a $k$-form $\omega$ on $\HH^{\ast}_3$ to be a $k$-form $\omega_0$ on $\HH_3$ and a family $k$-forms $\omega_{r}$ on $\Hcal_r$ such that $\res_r(\omega_0)=\omega_r$. We denote by $\Omega^k(\HH^\ast_3)$ the space of such forms. Let 
\begin{align}
\Omega^k(\HH^\ast_3;\C) \coloneqq \Omega^k(\HH^\ast_3) \otimes_\R \C    
\end{align} be the differential forms valued in $\C$. Let $\Omega^k(\HH^\ast_3;\C)^\Gamma$ be the complex of $\Gamma$-invariant forms, consisting of forms that satisfy $\gamma^\ast \omega_0=\omega_0$ and $\gamma^\ast \omega_{r}=\omega_{\gamma^{-1}r}$. Let $ H^k(X_\Gamma;\C) = H^k(\Omega^\bullet(\HH^\ast_3;\C)^\Gamma)$ be the cohomology of this complex. Similarly we have
\begin{align}
    H^k(\partial X_\Gamma; \C) =  \bigoplus_{c=[r] \in C_\Gamma} H^k (\Omega^\bullet(\Hcal_r;\C)^{\Gamma_r}) \simeq H^k (\Omega^\bullet(\partial X_\Gamma;\C)^{\Gamma}).
\end{align}
It follows from \eqref{equivariance} that the restriction map induces a map 
\begin{align} \label{resdef}
    \res \colon H^k(Y_\Gamma; \C) & \longrightarrow H^k(\partial X_\Gamma;\C) \nonumber \\
    \omega & \longmapsto (\res_r(\omega))_{c=[r]}.
\end{align}

\paragraph{Relative cohomology.} Let $\Omega^k(\HH_3^\ast,\partial \HH_3^\ast; \C)^\Gamma \coloneqq \Omega^{k}(\HH_3^\ast; \C)^\Gamma \oplus \Omega^{k-1}(\partial \HH_3^\ast;\C)^\Gamma$ be the complex with the coboundary operator
\begin{align}
    \delta \colon \Omega^k(\HH_3^\ast,\partial \HH_3^\ast; \C)^\Gamma & \longrightarrow \Omega^{k+1}(\HH_3^\ast,\partial \HH_3^\ast; \C)^\Gamma \nonumber \\
    (\omega,\theta) & \longmapsto (d\omega, \iota^\ast \omega-d \theta).
\end{align}
The cohomology $H^k(X_\Gamma,\partial X_\Gamma;\C)$ of $X_\Gamma$ relative to $\partial X_\Gamma$ is the cohomology associated to this complex. We have an exact sequence
\begin{align}
    0 \longrightarrow \Omega^{k-1}(\partial \HH_3^\ast;\C)^\Gamma \xhookrightarrow{\quad \alpha \quad} \Omega^k(\HH_3^\ast,\partial \HH_3^\ast; \C)^\Gamma \xrightarrow{\quad \beta \quad} \Omega^k(\HH_3^\ast;\C)^\Gamma \longrightarrow 0
\end{align}
where the first map is given by $\alpha(\theta)=(0,\theta)$ and the second by $\beta(\omega,\theta)=\omega$. This induces a long exact sequence in cohomology

\begin{equation} \label{les1}
\hspace*{-2.5cm}  \begin{tikzcd}
H^{k-1}(\partial X_\Gamma;\C) \rar["\alpha^\ast"] & H^k(X_\Gamma, \partial X_\Gamma;\C) \rar["\beta^\ast"] & H^k(X_\Gamma;\C) \ar["\res", out=-10, in=170]{dll} \\
H^{k}(\partial X_\Gamma;\C) \rar["\alpha^\ast"] & H^{k+1}(X_\Gamma, \partial X_\Gamma;\C) \rar["\beta^\ast"] & H^{k+1}(X_\Gamma;\C)
\end{tikzcd}    
\end{equation} and the boundary map $\res$ is the restriction to the boundary. 

\paragraph{Homology.} For an abelian group $A$ let $C_n(\HH_3;A)$ be the $\Z$-module of singular $n$-chains valued in $A$. The action of $\Gamma$ on $\HH_3$ endows $C_n(\HH_3;A)$ with a $\Gamma$-module structure and we define the complex of coinvariant chains
\begin{align}
    C_n(\HH_3;A)_\Gamma \coloneqq C_n(\HH_3;A)/ \bigl < \sigma - \gamma \sigma  \bigr >,
\end{align}
where we quotient by the submodule generated by all $\sigma - \gamma \sigma$ with $\sigma \in C_n(\HH_3;A)$. Let $H_n(Y_\Gamma;A)$ be the homology of this complex. We define $H_n(X_\Gamma;A)$ similarly.

Let $u_0 \in \HH_3$ be any basepoint and $[u_0,\gamma u_0]$ be a path joining $u_0$ and $\gamma u_0$ for $\gamma \in \Gamma$. The boundary of $[u_0,\gamma u_0]$ is $\gamma u_0-u_0$ hence it represents a class in $H_1(Y_\Gamma;\C)$.
\begin{prop} \label{surjectivegamma}
The map  $ \Gamma \longrightarrow H_1(Y_\Gamma;\Z)$ sending $\gamma$ to $[u_0,\gamma u_0]$ is a surjective morphism and independent of $u_0$.
\end{prop}
\begin{proof} If $Y_\Gamma$ were a manifold (for example if $\Gamma$ were some congruence subgroup of $\SL_2(\Ocal))$ we could work with singular homology. Then the map would be the Hurewicz homomorphism, which is surjective since $Y_\Gamma$ is path connected; see \cite[Theorem.~2A.1]{hatcher}.
In the case of equivariant homology it works almost in the same way and we follow the proof of \cite{hatcher}.

We write $[a,a']\sim [b,b']$ for two homologous paths joining points $a,a',b$ and $b'$ in $\HH_3$. We have the following relations:
\begin{itemize}
    \item[(1)] \quad $[a,a] \sim 0$,
    \item[(2)] \quad $[a,b]+[b,c] \sim [a,c]$
    \item[(3)] \quad $[a,b] \sim -[b,a]$
    \item[(4)] \quad  $[a,\gamma a] \sim [b , \gamma b ]$,
    \item[(5)] \quad $[a,b] \sim [\gamma a, \gamma b]$.
\end{itemize}
The first two ones are clear since $[a,a]$ is the boundary of the constant $2$-simplex $C=\{ a \}$ and $[a,b] - [a,c] +[b,c]$ is the boundary of the simplex joining $a,b$ and $c$. The third one follows from $(1)$ and $(2)$. For $(4)$ consider the two simplices $C_1$ and $C_2$ as in Figure \ref{simplicial1}. The boundaries are given by
\begin{align}
    \partial C_1 & = [\gamma a, \gamma b]-[a,\gamma b]+[a,\gamma a] \nonumber \\
    \partial C_2 & = [b,\gamma b ]-[a,\gamma b]+[a,b]
\end{align}
so that
\begin{align}
    \partial (C_1-C_2) & = [a,\gamma a]-[b;\gamma b] +[\gamma a, \gamma b ]-[a,b]= [a,\gamma a]-[b,\gamma b]
\end{align}
where $[\gamma a, \gamma b ]-[a,b]=0$ since we are in the complex of coinvariants. 
\begin{figure}[h] 
\centering
\begin{tikzpicture} 
\draw (0,0) -- (5,0) ;
\draw (0,0) -- (0,5) ;
\draw (0,0) -- (5,5) ;
\draw (5,0) -- (5,5) ;
\draw (0,5) -- (5,5) ;
\draw (0,0) node[below left]{$v_0=a$} ;
\draw (2,3) node[above]{$C_1$};
\draw (3,2) node[right]{$C_2$};
\draw (0,5) node[above left]{$v_1=\gamma a$} ;
\draw (5,0) node[below right]{$v_1=b$};
\draw (5,5) node[above right]{$v_1=\gamma b$};
\end{tikzpicture}
\captionsetup{width=.75\textwidth,font={small,it},justification=centering}
\caption{The equivalence $[a,\gamma a] \sim [b , \gamma b ]$.}
\label{simplicial1}
\end{figure}
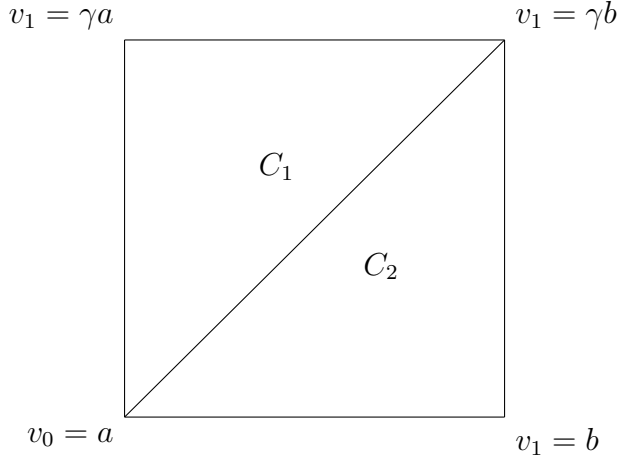
For $(5)$, consider the simplices $C_1$ and $C_2$ as in Figure \ref{simplicial2}. The boundaries are given by
\begin{align}
    \partial C_1 & = [b, \gamma a]-[\gamma b,\gamma a]+[\gamma b,b] \nonumber \\
    \partial C_2 & = [b,\gamma a ]-[a,\gamma a]+[a,b]
\end{align}
so that
\begin{align}
    \partial (C_1-C_2) & = [\gamma a,\gamma b]-[b,\gamma b] +[a, \gamma a ]-[a,b]= [\gamma a,\gamma b]-[a,b]
\end{align}
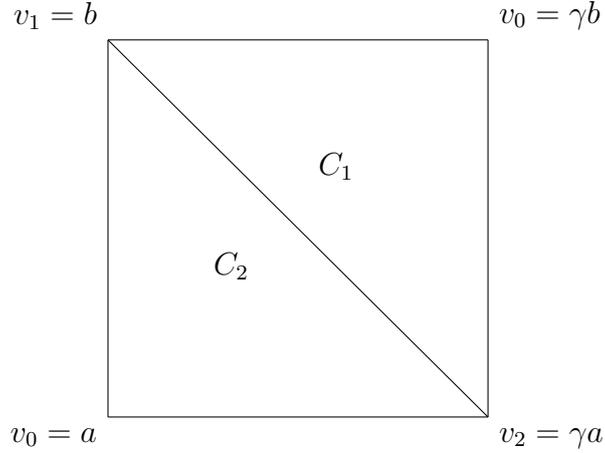
\begin{figure} [h]
\centering
\begin{tikzpicture} 
\draw (0,0) -- (5,0) ;
\draw (0,0) -- (0,5) ;
\draw (5,0) -- (0,5) ;
\draw (5,0) -- (5,5) ;
\draw (0,5) -- (5,5) ;
\draw (0,0) node[below left]{$v_0=a$} ;
\draw (3,3) node[above]{$C_1$};
\draw (2,2) node[left]{$C_2$};
\draw (0,5) node[above left]{$v_1=b$} ;
\draw (5,0) node[below right]{$v_2=\gamma a$};
\draw (5,5) node[above right]{$v_0=\gamma b$};
\end{tikzpicture}
\captionsetup{width=.75\textwidth,font={small,it},justification=centering}

\caption{The equivalence $[a,b] \sim [\gamma a , \gamma b ]$.}
\label{simplicial2}
\end{figure}
Let us now prove the statement of the proposition. The fact that it is independent of the basepoint follows from $(4)$, and the fact that it is a homorphism follows from $(2)$ and $(4)$ since
\begin{align}
    [u_0,\gamma_1 u_0]+[u_0,\gamma_2 u_0] \sim [\gamma_2 u_0,\gamma_1 \gamma_2 u_0]+[u_0,\gamma_2 u_0] \sim [u_0,\gamma_1 \gamma_2 u_0].
\end{align}
For the surjectivity, suppose we have a class represented by a cycle
\begin{align}
    \sigma = \sum_{i=1}^m n_i [a_i,b_i] \in C_1(\HH_3)_\Gamma.
\end{align}
After relabeling the paths we can suppose that $n_i=\pm 1$, and using $(3)$ we can even suppose that $n_i=1$. Since the boundary is 
\begin{align}
    \partial \sigma = \sum_i (b_i-a_i) =0,
\end{align}
we necessarily have that every $b_j$ is equal to $\gamma_{ij} a_i$ for a unique $a_i$ and some $\gamma_{ij} \in \Gamma$. We can see it as a permutation on the set $\{1, \dots, m\}$, where we send $i$ to $j$ if $a_i$ is $\Gamma$-equivalent to $b_j$. 

First suppose that the corresponding permutation is the identity {\em i.e}  $b_i=\gamma_i a_i$ for every $i$ . Then using $(4)$ and the fact that the map is a morphism we get
\begin{align}
\sigma = \sum_{i=1}^m [a_i,\gamma_i a_i] \sim \sum_{i=1}^m [u_0,\gamma_i u_0]\sim [u_0,\gamma_1 \cdots \gamma_m u_0].
\end{align}

Now suppose that the permutation contains some cycle of order $n$, which means that $\sigma$ contains a cycle
\begin{align}
    \sigma'= [a_1,\gamma_n a_n]+[a_2,\gamma_1 a_1]+ [a_3,\gamma_2 a_2] + \cdots + [a_n,\gamma_{n-1}a_{n-1}].
\end{align}
Using $(5)$ and $(2)$ we can sum the first two terms
\begin{align}
    [a_1,\gamma_n a_n]+[a_2,\gamma_1 a_1] \sim [\gamma_1 a_1, \gamma_1 \gamma_n a_n]+[a_2,\gamma_1 a_1] \sim [a_2,\gamma_1 \gamma_n a_n].
\end{align}
By induction we then get $\sigma' \sim [a_n, \gamma_{n-1} \gamma_{n-2} \cdots \gamma_1 \gamma_n a_n]$, so that we are reduced to the first case.
\end{proof}

\paragraph{Integral structure on the cohomology.} Let $\Rcal \subset \C$ be a torsion free $\Ocal$-submodule. We will later take $\Rcal=\Ocal_{F}$. We have a pairing
\begin{align}
    \bigl <  \quad, \quad \bigr > \colon H_1(Y_\Gamma; \C) \otimes H^1(Y_\Gamma; \C) & \longrightarrow \C \nonumber \\
    ([\sigma] , [\omega]  ) & \longmapsto \int_\sigma \omega
\end{align}
where $\omega \in \Omega^1(\HH_3;\C)^\Gamma$ and $\sigma \in C_1(\HH_3)_\Gamma$; see \cite[Satz.~3]{feldhusen}. Note that $\Rcal \otimes_\Rcal \C=\C$, so that we have a map
\begin{align}
    H_1(Y_\Gamma; \Rcal) \longrightarrow H_1(Y_\Gamma;\Rcal) \otimes_\Rcal \C = H_1(Y_\Gamma;\C).
\end{align}
Let $\cohomt_1(Y_\Gamma;\Rcal)$ be the image of this map. The kernel is the torsion part of $H_1(Y_\Gamma;\Rcal)$, so that we can identify $\cohomt_1(Y_\Gamma;\Rcal)$ with the free part of $H_1(Y_\Gamma;\Rcal)$. We use the pairing to define the cohomology groups
\begin{align}
\cohomt^1(Y_\Gamma; \Rcal) \coloneqq \left \{  \Bigl . [\omega] \in H^1(Y_\Gamma; \C) \Bigr  \vert \; \bigl < [\omega] , [\sigma] \bigr > \in \Rcal \quad \textrm{for all} \; [\sigma] \in \cohomt_1(Y_\Gamma; \Rcal) \right \}.
\end{align}
Since $\Rcal$ is torsion-free, we can identify $H_1(Y_\Gamma; \Rcal)$ with $H_1(Y_\Gamma;\Z) \otimes_{\Z} \Rcal$. Hence, by Proposition \ref{surjectivegamma} a class $[\omega]$ is in $\cohomt^1(Y_\Gamma; \Rcal)$ if and only if
\begin{align}
    \int_{u_0}^{\gamma u_0} \omega \in \Rcal
\end{align}
for all $\gamma \in \Gamma$.
\begin{rmk}The $\Rcal$-module $\cohomt^1(Y_\Gamma;\Rcal)$ is the torsion free part of the sheaf cohomology $\cohom^1(Y_\Gamma; \Rcal)$, that we identify with the image 
\begin{align}
    \im \left ( \cohom^1(Y_\Gamma; \Rcal) \longrightarrow \cohom^1(Y_\Gamma; \C) \right ).
\end{align}
\end{rmk}
\subsection{Cohomology of the boundary}

Recall that for any matrix $M$ in $\SL_2(K)$ sending $r$ to $\8$ we had a map \eqref{amap2}
\begin{align*} \label{amap3}
    \phi_{M,r} \colon \Gamma_r \backslash \Hcal_r & \longrightarrow \C/\afrak_M^{-2} \nonumber, \\
    \Gamma_r(z:1) & \longmapsto Mz
\end{align*}
where $\afrak_M=(m)+(n)$. Note that if $c=[r]$, then $[\afrak_M]=[\afrak_c]$ in the bijection \eqref{clbij} between $C_\Gamma$ and $\Cl(K)$.

\begin{lem} \label{boundforms} Let $\chi$ be an unramified Hecke character of infinity type $(-2,0)$. The forms \begin{align*}
    \omega_{\chi,r} & \coloneqq \chi(\afrak_M)^{-1}\phi_{M,r}^\ast dz \nonumber \\
    \bar{\omega}_{\chi,r} & \coloneqq \chi(\afrak_M)^{-1}\phi_{M,r}^\ast d\bar{z}
\end{align*}
lie in $\cohomt^1(\Gamma_r \backslash \Hcal_r,\Ocal_{F})$ and do not depend on the choice of $M$. Furthermore, we have $\gamma^\ast \omega_{\chi,r}=\omega_{\chi,\gamma^{-1} r}$ and $\gamma^\ast \bar{\omega}_{\chi,r}=\bar{\omega}_{\chi,\gamma^{-1} r}$.
\end{lem}
\begin{proof} We prove the statements for $\omega_{\chi,r}$, they are similar for $\bar{\omega}_{\chi,r}$.
    Since $\afrak_M^{-2}$ is the period lattice of the elliptic curve $\C/\afrak_M^{-2}$, we have
    \[\int_\gamma dz \in \afrak_M^{-2}\]
    for any $\gamma \in H_1(\C/\afrak_M^{-2},\Z)$. After tensoring with $\Ocal_{F}$ and recalling from the proof of Proposition \ref{proplambda} that $\chi(\afrak_M)\Ocal_{F}=\afrak_M^{-2}\Ocal_{F}$, we get that
    \begin{align}\int_\gamma \chi(\afrak_M)^{-1}dz \in \chi(\afrak_M)^{-1}\afrak_M^{-2}\Ocal_{F} \subset \Ocal_{F}.\end{align}
    It follows that
    \begin{align}\chi(\afrak)^{-1}dz \in \cohomt^1(\C/\afrak_M^{-2},\Ocal_{F}).\end{align}
    Hence the pullback $\omega_{\chi,r}$ is also integral.
    Now suppose that $N$ is another matrix in $\SL_2(K)$ sending $r$ to $\8$. Then $M=PN$ where
    $P=\begin{pmatrix}
        a & b \\ 0 & a^{-1}
    \end{pmatrix}$.
    We have $\afrak_N=a^{-1}\afrak_M$ and the following diagram commutes
    \begin{equation}
\begin{tikzcd}
  \Gamma_r \backslash \Hcal_r \arrow[rd,"\Phi_{N,r}",swap] \arrow[r,"\Phi_{M,r}"] & \C/\afrak_M^{-2} \arrow[d] \\
 & \C/a^2\afrak_M^{-2},
\end{tikzcd}
\end{equation}
where the vertical map sends $z$ to $Nz=a^2z+ab$. Hence $N^\ast dz=a^2dz$ and
\begin{align}
    \chi(\afrak_N^{-2})\phi^\ast_{N,r}dz=a^{-2}\chi(\afrak_M^{-2})\phi^\ast_{M,r}N^\ast dz=\chi(\afrak_M^{-2})\phi^\ast_{M,r}dz.
\end{align}
Finally, let $\gamma$ be in $\Gamma$. We have seen that the form does not depend on the choice of $M$, so we can take $M\gamma$ to be the matrix in $\SL_2(K)$ sending $\gamma^{-1} r $ to $\8$. We have $\afrak_{M\gamma^{-1}}=\afrak_M$ and the following diagram commutes
    \begin{equation}
\begin{tikzcd}
  \Gamma_r \backslash \Hcal_r \arrow[d] \arrow[r,"\Phi_{M,r}"] & \C/\afrak_M^{-2} \arrow[d] \\
  \Gamma_{\gamma^{-1} r} \backslash \Hcal_{\gamma^{-1} r}  \arrow[r,"\Phi_{M\gamma,r}"] & \C/\afrak_M^{-2}.
\end{tikzcd}
\end{equation}
Thus
\begin{align}
    \omega_{\chi,\gamma^{-1}r}=\chi(\afrak_{M\gamma})^{-1}\phi_{M\gamma,r}^\ast dz =\chi(\afrak_{M})^{-1} \gamma^\ast \phi_{M,r}^\ast dz=\omega_{\chi,r}.
\end{align}
\end{proof}

The complex conjugation on $\C$ induces an involution on $\HH_3$
\begin{align}
    \iota \colon \HH_3 & \longrightarrow \HH_3 \nonumber \\
    z+jv & \longmapsto \bar{z}+jv.
\end{align}
It extends canonically to $\partial \HH_3$ by sending $z \in \Hcal_r$ to $\bar{z} \in \Hcal_{\bar{r}}$. Consider the involution $I(\gamma)=\bar{\gamma}$ on $\Gamma$. One can check that
\begin{align} \label{involutions}
    \iota \circ \gamma (u)=I(\gamma) \circ \iota(u).
\end{align} Hence, the involution $\iota$ descends to an involution on $X_\Gamma$ and restricts to an involution on $\partial X_\Gamma$. The pullback of differential forms by these involutions induce compatible involutions on $H^1(X_\Gamma;\C)$ and $H^1(\partial X_\Gamma;\C)$. At the level of the boundary forms we have
\begin{align}\label{involution2}
    \iota^\ast \omega_{\chi,r}=\bar{\omega}_{\chi,\bar{r}}.
\end{align}
Let $H^1(\partial X_\Gamma;\C)^{-}$ be the $(-1)$-eigenspace of this involution.

\begin{prop}  \label{imres} We have
\[\dim \im(\res) =\frac{1}{2}  \dim_\C H^1(\partial X_\Gamma;\C)= h.\]
More precisely, the map
\[H^1 (X_\Gamma,\C ) \longrightarrow H^1(\partial X_\Gamma,\C)^-\]
is surjective.
\end{prop}
\begin{proof} The result follows from a theorem of Serre. We refer to \cite[Proposition.~24, Corollary.~26]{berger_2009} for a proof. However, let us prove the statement about the dimension. Let
\begin{align}
    \alpha^\ast \colon H^1(\partial X_\Gamma; \C) & \longrightarrow H^2(X_\Gamma, \partial X_\Gamma;\C)
\end{align}
be the map from the long exact sequence, so that $\ker(\alpha^\ast)=\im(\res)$. By Poincaré duality, we have 
\begin{align}
H^1(\partial X_\Gamma;\C)^\vee \simeq H^1(\partial X_\Gamma; \C)    
\end{align} and 
\begin{align}
    H^2(X_\Gamma, \partial X_\Gamma;\C) \simeq H^2_c(Y_\Gamma;\C) \simeq H^1(Y_\Gamma;\C)^\vee
\end{align} so that we can see $\alpha^\ast$ as a map
\begin{align}
    \alpha^\ast \colon H^1(\partial X_\Gamma;\C)^\vee & \longrightarrow H^1(Y_\Gamma;\C)^\vee
\end{align}
Since for $\theta \in \Omega^1(\partial X_\Gamma)$ and $\omega \in H^1(X_\Gamma,\partial X_\Gamma)$ we have
\begin{align}
    \int_{X_\Gamma} \alpha(\theta) \wedge \omega = \int_{\partial X_\Gamma} \theta \wedge \res(\omega),
\end{align}
it follows that $\alpha^\ast=\res^\vee$ is adjoint to $\res$. Moreover, the space $\im(\res)$ is an isotropic subspace since
\begin{align}
    \im(\res)=\ker(\alpha^\ast)=\ker(\res^\vee)=\im(\res)^\perp.
\end{align}
 Thus, it must be of half the dimension of the total space.
\end{proof}

\subsection{Eisenstein map}
In \eqref{resdef} we defined a restriction map
\begin{align}
    \res \colon H^1(Y_\Gamma; \C) \longrightarrow H^1(\partial X_\Gamma; \C).
\end{align}
The kernel is the interior cohomology $ H^1_!(Y_\Gamma; \C)$ and can be identified with the image of the compactly supported cohomology inside $H^1(Y_\Gamma; \C)$.  We will define an Eisenstein map 
\begin{align}
    \Eis \colon H^1(\partial X_\Gamma;\C) \longrightarrow H^1(Y_\Gamma;\C)
\end{align}
whose image will be the Eisenstein cohomology $H^1_{\Eis}(Y_\Gamma;\C)$.
\begin{rmk}
 As it will follow from Proposition \ref{resteis}, the map $\Eis$ is not a section of $\res$, {\em i.e.} we do not have $\res \circ \Eis = \id$.
\end{rmk} We begin by defining a map  
\begin{align}
    \Eis \colon \cohom^1(\Gamma_\8\backslash \Hcal_\8;\C) \longrightarrow \cohom^1( Y_\Gamma; \C)
\end{align}
at the cusp $\8$. We have a $\Gamma_\8$-equivariant map
\begin{align} \label{projection}
    p_\8 \colon  \HH_3 & \longrightarrow \Hcal_\8 \nonumber, \\
   z+jv & \longmapsto (z:1),
\end{align}
that we can use to pull back a form $\omega_\8$ in $\Omega^1(\Hcal_\8)^{\Gamma_\8}$ to a form 
\begin{align}
    p_\8^\ast \omega_\8 \in \Omega^1(\HH_3)^{\Gamma_\8}.
\end{align}
To obtain a form on $Y_\Gamma$ we define
\begin{align} \label{eisstep1}
    \Eis(\omega_\8) \coloneqq \sum_{\gamma \in \Gamma_\8 \backslash \Gamma} \gamma^\ast p^\ast \omega \in \Omega^1(\HH_3)^{\Gamma}.
\end{align}
Similarly, at the other cusps $r$ we can define the $\Gamma_r$-equivariant map
\begin{align} \label{projection2}
    p_{r} \colon  \HH_r & \longrightarrow \Hcal_r
\end{align} to be the composition $p_r=M^{-1} \circ p_\8 \circ M$:
\begin{equation}
\begin{tikzcd}
 \HH_3 \arrow[d,"M"] \arrow[dotted,r,"p_{r}"] & \Hcal_r  \\
  \HH_3  \arrow[r,"p_\8"] & \Hcal_\8 \arrow[u,"M^{-1}"],
\end{tikzcd}
\end{equation}
where $M$ is a matrix in $\SL_2(K)$ sending $r$ to $\8$ as earlier. Note that $p_r$ does not depend on the choice of $M$ since
\begin{align}
    \begin{pmatrix}
        a & b \\ 0 & a^{-1}
    \end{pmatrix}^{-1} \circ p_\8 \circ \begin{pmatrix}
        a & b \\ 0 & a^{-1}
    \end{pmatrix} = p_\8.
\end{align}
For a form $\omega_r$ in $\Omega^1(\Hcal_r)^{\Gamma_r}$ we define
\begin{align} \label{eisstep2}
    \Eis(\omega_r)= \sum_{\gamma \in \Gamma_r \backslash \Gamma} \gamma^\ast p_r^\ast \omega_r \in \Omega^1(\HH_3)^\Gamma.
\end{align}
 However, the sums \eqref{eisstep1} and \eqref{eisstep2} are not convergent, and need to be regularized.

\subsection{Regularization of the Eisenstein series}
Since $\omega_{\chi,r}$ and $\bar{\omega}_{\chi,r}=\iota^\ast \omega_{\chi,\bar{r}}$ span the space of forms on $\Gamma_r\backslash \Hcal_r$, it is enough to regularize
\begin{align}
    E_{\chi,c} \coloneqq \Eis(\omega_{\chi,r}) \in \Omega^{1}(\HH_3)^{\Gamma},
\end{align}
where $c=[r]$ is a cusp. Note that the left hand side only depends on $c$ since for $\gamma$ in $\Gamma$ we have
\begin{align}
    \Eis(\omega_{\chi,\gamma^{-1}r})= \Eis(\gamma^\ast \omega_{\chi,r})=\Eis(\omega_{\chi,r}).
\end{align}
In particular, we will frequently denote this form by $E_{\chi,\afrak}$ where $\afrak$ is in the ideal class of $\afrak_c$, corresponding to the cusp $c$.

Let $r=(m:n)$ be represented by coprime integers and let $M$ in $\SL_2(K)$ be the matrix as earlier, sending $r$ to $\8$. We have:
\begin{align}
    E_{\chi,c} & = \chi(\afrak_M)^{-1}\sum_{\gamma \in \Gamma_r \backslash \Gamma} \gamma^\ast M^\ast dz \nonumber \\
    & = \chi(\afrak_M)^{-1}\sum_{\gamma \in (M\Gamma_rM^{-1}) \backslash M\Gamma} \gamma^\ast dz \nonumber \\
    & = \chi(\afrak_M)^{-1}\sum_{\gamma \in \Gamma(\afrak_M^{-2}) \backslash M\Gamma} \alpha^\ast dz.
\end{align}
Since $[\afrak_M]=[\afrak_c]$, we have
\begin{align}
    E_{\chi,\afrak} & = \chi(\afrak)^{-1}\sum_{\gamma \in \Gamma(\afrak^{-2}) \backslash M\Gamma} \alpha^\ast dz.
\end{align} 
\begin{lem} The map
    \begin{align}
    \Gamma(\afrak_M^{-2})_\8 \backslash M \Gamma & \longrightarrow \left \{ \left . (c,d) \in \afrak_M \times \afrak_M \right \vert (c)+(d)=\afrak_M \right \} \nonumber \\
    \alpha = \begin{pmatrix} a & b \\ c& d \end{pmatrix} & \longmapsto (c,d)=(0,1)\alpha
\end{align}
is a bijection.
\end{lem}
\begin{proof}
    First we have a bijection \begin{align}
    (\pm \Gamma_\8) \backslash \Gamma & \longrightarrow  (1,0) \Gamma  \nonumber \\
    \gamma=\begin{pmatrix} a & b \\ c & d \end{pmatrix} & \longmapsto (c,d)=(0,1)\gamma.
\end{align}
After moving the orbit by $M=\begin{pmatrix}
    y & -x \\ -n & m
\end{pmatrix}$ this becomes a bijection \begin{align}
    (\pm\Gamma(\afrak^{-2})_\8) \backslash M \Gamma & \longrightarrow (0,1)M \Gamma= (-n,m)\Gamma\nonumber \\
    \alpha=\begin{pmatrix} a & b \\ c & d \end{pmatrix} & \longmapsto (0,1)\alpha=(c,d)
\end{align}
where $\alpha$ is in $M\Gamma$. Moreover, the orbit $(-n:m)\Gamma$ is in bijection with the set 
\begin{align}
\left \{ \left . (c,d) \in \afrak_M \times \afrak_M \right \vert (c)+(d)=\afrak_M \right \}/\Ocal^\times
\end{align}
where $\Ocal^\times=\{ \pm 1\}$ acts diagonally on $\afrak_M \times \afrak_M$. First, for $\gamma \in \Gamma$ we have $(c,d)=(-n,m) \gamma \in \afrak_M \times \afrak_M.$
Furthermore, since the rows of $\gamma$ are coprime we have $(c) +(d)= (n)+(m)= \afrak_M$.
Conversely, suppose that we have a pair $(c,d) \in \afrak_M \times \afrak_M$ such that $(c)+(d)=\afrak_M$. Hence $(c)\afrak_M^{-1}+(d)\afrak_M^{-1}=\Ocal$ and we can find $(b,a) \in \afrak_M^{-1} \times \afrak_M^{-1}$ such that $bc-ad=1$. We then have a matrix
\begin{align}
    \begin{pmatrix}
        a & b \\ c & d
    \end{pmatrix} \in \begin{pmatrix}
       \afrak_M^{-1}  & \afrak_M^{-1} \\ 
       \afrak_M & \afrak_M 
    \end{pmatrix} \cap \SL_2(K).
\end{align}
Then
\begin{align}
    M^{-1}\begin{pmatrix}
        a & b \\ c & d
    \end{pmatrix} \in \begin{pmatrix}
        \afrak_M  & \afrak_M^{-1} \\ 
        \afrak_M & \afrak_M^{-1} 
    \end{pmatrix} \begin{pmatrix}
       \afrak_M^{-1}  & \afrak_M^{-1} \\ 
       \afrak_M & \afrak_M 
    \end{pmatrix} \cap \SL_2(K) = \Gamma,
\end{align}
and we can write
\begin{align}
    (c,d)=(0,1)\begin{pmatrix}
        a & b \\ c & d
    \end{pmatrix} & =(0,1)M \left ( M^{-1} \begin{pmatrix}
        a & b \\ c & d
    \end{pmatrix}\right ) \nonumber \\
    & =(-n,m)\left ( M^{-1} \begin{pmatrix}
        a & b \\ c & d
    \end{pmatrix}\right ) \in (-n,m) \Gamma.
\end{align}
Composing the two bijections we then have
\begin{align}
    (\pm\Gamma(\afrak_M^{-2})_\8) \backslash M \Gamma & \longrightarrow \left \{ \left . (c,d) \in \afrak_M \times \afrak_M  \right \vert (c)+(d)=\afrak_M\right \}/\Ocal^\times \nonumber \\
    \gamma = \begin{pmatrix} a & b \\ c & d \end{pmatrix} & \longmapsto (c,d).
\end{align}
The lemma follows from the observation that the action of $\pm 1$ on the left hand side correspond to the action of $\Ocal^\times$ on the right hand side.
\end{proof}

For $u=z+jv \in \HH_3$ let $z(u)=z$, $\bar{z}(u)={\bar{z}}$ and $v(u)=v$ be the coordinate functions. Let $\alpha \in \SL_2(\C)$ and 
\begin{align}
\eta(u,c,d) & \coloneqq \alpha^\ast(dz)    
\end{align} where $\alpha = \begin{pmatrix} a & b \\ c & d \end{pmatrix}$. It follows from \eqref{involutions} that
\begin{align}
\iota^\ast \eta(u,\overline{c},\overline{d}) = \iota^\ast I(\alpha)^\ast(dz)=\alpha^\ast (\iota^\ast dz) =\alpha^\ast (d \bar{z}).  
\end{align}
The following calculations will indeed show that $\eta$ and $\overline{\eta}$ depend only on $c$ and $d$. We have
\begin{align}
    z(\alpha u) = \frac{(az+b)\overline{(cz+d)}+a\overline{c}v^2}{\vert cz+d \vert^2+\vert c v\vert^2},
\end{align}
and
\begin{align} \label{gammav}
    v(\alpha u) = \frac{v}{\vert cz+d \vert^2+\vert c v\vert^2}.
\end{align}
We view $dz$ as the differential of the coordinate map $z(u)$, hence
\begin{align}
\eta(u,c,d) & =\eta(u,c,d)_zdz+\eta(u,c,d)_{\bar{z}}d\bar{z}+\eta(u,c,d)_v dv
\end{align} where
\begin{align} \label{jacobian}
    \eta(u,c,d)_z & = \frac{\partial z(\alpha u)}{\partial z} = \frac{\left ( \overline{cz+d}\right )^2}{\left (\vert cz+d \vert^2+\vert c v \vert^2 \right )^2}, \nonumber \\
    \eta(u,c,d)_{\bar{z}} & = \frac{\partial z(\alpha u)}{\partial \bar{z}} = \frac{-(\overline{c}v)^2}{\left (\vert cz+d \vert^2+\vert c v \vert^2 \right )^2}, \\
    \eta(u,c,d)_v & = \frac{\partial z(\alpha u)}{\partial v} = 2\frac{\left ( \overline{cz+d}\right )\overline{c}v}{\left (\vert cz+d \vert^2+\vert c v \vert^2 \right )^2}. \nonumber
\end{align} If $J((c,d),u)$ is the Jacobian
    \begin{align} \label{jacobian3}
        J(\alpha,u) \coloneqq \begin{pmatrix}
            \frac{\partial z(\alpha u)}{\partial z} & \frac{\partial z(\alpha u)}{\partial \bar{z}} & \frac{\partial z(\alpha u)}{\partial v} \\[0.5cm]
            \frac{\partial \bar{z}(\alpha u)}{\partial z} & \frac{\partial \bar{z}(\alpha u)}{\partial \bar{z}} & \frac{\partial \bar{z}(\alpha u)}{\partial v} \\[0.5cm]
            \frac{\partial v(\alpha u)}{\partial z} & \frac{\partial v(\alpha u)}{\partial \bar{z}} & \frac{\partial v(\alpha u)}{\partial v}
        \end{pmatrix},
    \end{align}
then
\begin{align} \label{jacobian2}
 \eta(u,c,d)=(1,0,0)J(\alpha,u) \begin{pmatrix} dz \\ d\bar{z} \\ dv
     \end{pmatrix}.
\end{align} 
It follows from the previous lemma
\begin{align}
    E_{\chi,\afrak}=\chi(\afrak)^{-1}\sum_{{\substack{(c,d) \in \afrak \times \afrak \\ (c)+(d)=\afrak_M}}} \eta(u,c,d).
\end{align}
This sum does not converge, and in order to regularize it we define for a complex number $s$:
\begin{align}
 \eta(u,c,d,s) & \coloneqq \alpha^\ast \left ( v^sdz \right ) \nonumber \\
 & =(1,0,0)J(\alpha,u)v(\alpha u)^s \begin{pmatrix} dz \\ d\bar{z} \\ dv
     \end{pmatrix}.   
\end{align}
More precisely we have
\begin{align}
    \eta(u,c,d,s) = \eta(u,c,d,s)_zdz+\eta(u,c,d,s)_{\bar{z}}d\bar{z}+\eta(u,c,d,s)_vdv
\end{align}
with
\begin{align}
    \eta(u,c,d,s)_z & \coloneqq v^{s}\frac{\left ( \overline{cz+d}\right )^2}{\left (\vert cz+d \vert^2+\vert c v \vert^2 \right )^{2+s}}, \nonumber \\
    \eta(u,c,d,s)_{\bar{z}} & \coloneqq v^{s}\frac{-(\overline{c}v)^2}{\left (\vert cz+d \vert^2+\vert c v \vert^2 \right )^{2+s}}, \\
    \eta(u,c,d,s)_v & \coloneqq  2v^{s}\frac{\left ( \overline{cz+d}\right )\overline{c}v}{\left (\vert cz+d \vert^2+\vert c v \vert^2 \right )^{2+s}} .\nonumber
\end{align}
For an ideal $\afrak$ and an unramified Hecke character $\chi$ of infinity type $(-2,0)$ we define
\begin{align} \label{firsteis}
    E_{\chi,\afrak} (u,s) & \coloneqq \chi(\afrak)^{-1}\N(\afrak)^s\sum_{{\substack{(c,d) \in \afrak \times \afrak \\ (c)+(d)=\afrak}}} \eta(u,c,d,s).
\end{align}
Let us also define the forms
\begin{align}
      \Ehat_{\chi,\afrak}(u,s) & \coloneqq \chi(\afrak)^{-1}\N(\afrak)^s \sum_{(c,d) \in \afrak \times \afrak} \eta(u,c,d,s),
\end{align} 
that appear in \cite{ito,bcg21}.
Note that here we sum over $\afrak \times \afrak$ instead of the subset with $(c)+(d)=\afrak$ as is \eqref{firsteis}. We have
\begin{align}
    \Ehat_{\chi,\afrak}(u,s)=\Ehat_{\chi,\afrak,z}(u,s)dz+\Ehat_{\chi,\afrak,\bar{z}}(u,s)d\bar{z}+\Ehat_{\chi,\afrak,v}(u,s)dv
\end{align}
where 
\begin{align}
    \Ehat_{\chi,\afrak,z}(u,s) & \coloneqq  \chi(\afrak)^{-1}\N(\afrak)^s\sum_{(c,d) \in \afrak \times \afrak} \eta(u,c,d,s)_z,
\end{align}
and similarly for $\Ehat_{\chi,\afrak,\bar{z}}(u,s)$ and $\Ehat_{\chi,\afrak,v}(u,s)$. By \eqref{jacobian2} we have
\begin{align} \label{Evector}
\left ( \Ehat_{\chi,\afrak,z}(u,s), \Ehat_{\chi,\afrak,\bar{z}}(u,s), \Ehat_{\chi,\afrak,v}(u,s) \right ) & = \chi(\afrak)^{-1}\N(\afrak)^s\sum_{(c,d) \in \afrak \times \afrak} (1,0,0)J(\alpha,u) v(\alpha u)^s.
\end{align}
Since  $\eta(u,\kappa c, \kappa d,s)=\kappa^{-2}\vert \kappa \vert^{-2s}\eta(u,c,d,s)$ and $\chi((\kappa))^{-1}\N(\kappa)^s=\kappa^2\vert \kappa \vert^{2s}$, the forms $E_{\chi,\afrak}(u,s)$ and $\Ehat_{\chi,\afrak}(u,s)$
do not depend on the choice of the representative $\afrak$ of the class $[\afrak]$.

\begin{prop}\label{eisensteinrelation} We have
\[\Ehat_{\chi,\afrak}(u,s) = \sum_{i=1}^h \frac{\chi(\afrak_i\afrak^{-1})}{w(\afrak_i^{-1}\afrak)\N(\afrak_i\afrak^{-1})^{s}}G(1+s,2,0,0;\afrak_i^{-1}\afrak)E_{\chi, \afrak_i}(u,s).
\]  
\end{prop}
\begin{proof}
First we have
\begin{align}
    \Ehat_{\chi,\afrak}(u,s) & = \frac{\N(\afrak)^s}{\chi(\afrak)}\sum_{(c,d) \in \afrak \times \afrak} \eta(u,c,d,s) \nonumber \\
    & = \sum_{0 \neq \cfrak \subseteq \Ocal} \frac{\chi(\cfrak)}{\N(\cfrak)^s}\frac{\N(\afrak\cfrak)^s}{\chi(\afrak\cfrak)} \sum_{\substack{(c,d) \in \afrak \times \afrak \\ (c)+(d)=\afrak\cfrak}} \eta(u,c,d,s).
\end{align}
Note that if $(c,d) \in \afrak \times \afrak$ with $(c)+(d)=\afrak\cfrak$, then $(c,d) \in \afrak \cfrak \times \afrak \cfrak$. Hence
\begin{align}
    \Ehat_{\chi,\afrak}(u,s) & = \sum_{0 \neq \cfrak \subseteq \Ocal} \frac{\chi(\cfrak)}{\N(\cfrak)^{s}} E_{\chi, \cfrak\afrak}(u,s).
\end{align}
We write $\cfrak\afrak=\afrak_i(\alpha)$ for some representative $\afrak_i$ of the class group. Then $E_{\chi, \cfrak\afrak}(u,s)=E_{\chi, \afrak_i}(u,s)$ and we get
\begin{align}
    \Ehat_{\chi,\afrak}(u,s) & = \sum_{i=1}^h \frac{1}{w(\afrak_i^{-1}\afrak)}\frac{\chi(\afrak_i\afrak^{-1})}{\N(\afrak_i\afrak^{-1})^{s}} \sum_{0 \neq \alpha \subseteq \afrak_i^{-1}\afrak} \frac{1}{\alpha^2\vert \alpha \vert^{2s}} E_{\chi, \afrak_i}(u,s).
\end{align}
\end{proof}

\begin{prop} \label{continuation}
The series $\Ehat_{\chi,\afrak}(u,s)$ and $E_{\chi,\afrak} (u,s)$ converge for $\re(s) \gg 0$ and admit an analytic continuation to the whole plane. Moreover, at $s=0$, the forms $\Ehat_{\chi,\afrak}(u) \coloneqq \restr{\Ehat_{\chi,\afrak}(u,s)}{s=0}$ and $E_{\chi,\afrak}(u)\coloneqq \restr{E_{\chi,\afrak}(u,s)}{s=0}$ at $s=0$ are closed and are related by
\[\Ehat_{\chi,\afrak}(u) = \sum_{i=1}^h \frac{\chi(\afrak_i\afrak^{-1})}{w(\afrak_i^{-1}\afrak)}G_2(\afrak_i^{-1}\afrak)E_{\chi, \afrak_i}(u).
\] 
\end{prop}
\begin{proof}
The analytic continuation of $\Ehat_{\chi,\afrak}(u,s)$ can be done by Poisson summation, see for example \cite[page.~18]{bcg21}. The fact that the forms $\Ehat_{\chi,\afrak}(u,s)$ are closed is the content of \cite[Proposition.~3.3]{bcg21}. The same results holds for $E_{\chi,\afrak}$ by the previous proposition.
\end{proof}

\paragraph{The Eisenstein operator.} It follows from \eqref{involutions} and \eqref{involution2} that
\begin{align}
    \Eis(\bar{\omega}_{\chi,r})=\Eis(\iota^\ast\omega_{\chi,\bar{r}})=\iota^\ast \Eis(\omega_{\chi,\bar{r}})=\iota^\ast E_{\chi,c},
\end{align}
where $\iota$ is the involution induced by complex conjugation. The cohomology $H^1(\Gamma_r\backslash\Hcal_r;\C)$ is spanned by $\omega_{\chi,r}$ and $\bar{\omega}_{\chi,r}$.
Hence if $\epsilon_r=\alpha\omega_{\chi,r}+\beta\bar{\omega}_{\chi,r}$ we have
\begin{align}
    \Eis(\epsilon_r)=\alpha E_{\chi,c}+\beta \iota^\ast E_{\chi,c}.
\end{align}
Since $H^1(\partial X_\Gamma;\C)=\bigoplus_{c=[r] \in C_\Gamma} H^1(\Gamma_r\backslash\Hcal_r;\C)$, we have a map
\begin{align} \label{eisdef2}
        \Eis \colon \cohom^1(\partial X_\Gamma;\C) & \longrightarrow \cohom^1(Y_\Gamma;\C) \nonumber \\
        \sum_{c=[r] \in C_\Gamma}\lambda_{c}\epsilon_{r} & \longmapsto \sum_{c=[r] \in C_\Gamma} \lambda_{c} (\alpha E_{\chi,c}+\beta \iota^\ast E_{\chi,c}).
\end{align}
This the map \eqref{eismapintro} in the introduction.

\subsection{Fourier expansions and constant terms} 
\paragraph{At the cusp $\8$.} By Ito's computation \cite[p.~152]{ito}, we have the following Fourier expansions of the Eisenstein series:

\begin{align}
    \Ehat_{\chi,\afrak,z}(u,s)= & \chi(\afrak)^{-1} \N(\afrak)^sv^sG(1+s,2,0,0;\afrak) \nonumber \\
    & -\frac{2i(2\pi)^{s+2}}{\Dcal(\afrak)\Gamma(s+2)}v\sum_{\substack{(c,d)\in \afrak \times \afrak^{\vee} \\ cd \neq 0}}d^2 \left \vert \frac{d}{c} \right \vert^{s-1}K_{s-1}(4\pi \vert cd \vert v)e(cdz) 
\end{align}
and
\begin{align}
    \Ehat_{\chi,\afrak,\bar{z}}(u,s)= & -\frac{\chi(\afrak)^{-1} \N(\afrak)^s2i\pi}{D(\afrak)}G(s,2,0,0;\afrak) \nonumber \\
    & -\frac{2i(2\pi)^{s+2}}{\Dcal(\afrak)\Gamma(s+2)}v\sum_{\substack{(c,d)\in \afrak \times \afrak^{\vee} \\ cd \neq 0}}\bar{c}^2 \left \vert \frac{d}{c} \right \vert^{s+1}K_{s+1}(4\pi \vert cd \vert v)e(cdz), 
\end{align}
where $K_\nu(t)$ is the $K$-Bessel function, $e(z) \coloneqq \exp(2i\pi(z+\bar{z}))$, and $L^{\vee} \subset \C$ is the lattice dual to a lattice $L$. In particular at $s=0$ the constant term of $\Ehat_{\chi,\afrak,z}(u)$ is $\chi(\afrak)^{-1} G_2(\afrak).$
By the functional equation \eqref{funceq2}, we find that the constant term of $\Ehat_{\chi,\afrak,\bar{z}}(u)$ is
\begin{align}
   -\chi(\afrak)^{-1}G(\afrak)=-\chi(\afrak)^{-1}G_2(\afrak).
\end{align}

\paragraph{At other cusps.} Let $M$ be a matrix in $\SL_2(K)$ such that $Mr= \8$, as before. In this section let $N$ denote the inverse of $M$. Hence
\begin{align}
    N=M^{-1}= \begin{pmatrix}
    m & x \\
    n & y
    \end{pmatrix} \in \begin{pmatrix}
    \afrak_M & \afrak_M^{-1} \\
    \afrak_M & \afrak_M^{-1}
    \end{pmatrix}
\end{align}
where $\afrak_M=(m)+(n)$.
    Since
    \begin{align}
      \Ehat_{\chi,\afrak}(u)=(\Ehat_{\chi,\afrak,z}(u),\Ehat_{\chi,\afrak,\bar{z}}(u),\Ehat_{\chi,\afrak,v}(u)) \begin{pmatrix}
          dz \\ d\bar{z} \\ dv
      \end{pmatrix},
    \end{align}
    we have
    \begin{align}
      (N^\ast \Ehat_{\chi,\afrak}(u)=(\Ehat_{\chi,\afrak,z}(Nu),\Ehat_{\chi,\afrak,\bar{z}}(Nu),\Ehat_{\chi,\afrak,v}(Nu))J(N,u) \begin{pmatrix}
          dz \\ d\bar{z} \\ dv
      \end{pmatrix},
    \end{align} where $J(N,u)$ is the Jacobian as in \eqref{jacobian3}. We have  \begin{align}
        (\Ehat_{\chi,\afrak,z}(Nu),\Ehat_{\chi,\afrak,\bar{z}}(Nu),\Ehat_{\chi,\afrak,v}(Nu))= \chi(\afrak)^{-1} \sum_{\alpha \in \afrak \times \afrak} (
        1, 0, 0)J(\alpha,Nu),
    \end{align}
    where $\alpha=\begin{pmatrix}
        \ast & \ast \\ c & d
    \end{pmatrix}$. Since $J(\alpha,Nu)J(N,u)=J(\alpha N,u)$, we have
    \begin{align}
    N^\ast \Ehat_{\chi,\afrak}(u) = \chi(\afrak)^{-1} \sum_{(c,d) \in \afrak \times \afrak} (
        1, 0, 0)J(\alpha N,u) \begin{pmatrix}
          dz \\ d\bar{z} \\ dv
      \end{pmatrix}.   
    \end{align}
Let us define
    \begin{align}
      (\Ehat^{(N)}_z(u,s),\Ehat^{(N)}_{\bar{z}}(u,s),\Ehat^{(N)}_v(u,s))\coloneqq \chi(\afrak)^{-1} \N(\afrak)^s\sum_{(c,d) \in \afrak \times \afrak} (
        1, 0, 0)J(\alpha N,u)v(\alpha Nu)^s,
    \end{align}
    so that at $s=0$
    \begin{align}
      N^\ast \Ehat_{\chi,\afrak}(u) = (\Ehat^{(N)}_z(u),\Ehat^{(N)}_{\bar{z}}(u),\Ehat^{(N)}_v(u)) \begin{pmatrix}
          dz \\ d\bar{z} \\ dv
      \end{pmatrix}.
    \end{align}
Since these series only appear in this section, we drop the indices $\chi$ and $\afrak$ to simplify the notation. For $p$ in $\afrak \afrak_M$ we define $L_p \subset \C$ be the set of complex numbers $q=cx+dy$ where $(c,d) \in \afrak \times \afrak $ and $cm+dn=p$. By \cite[p.~162]{ito} there is a complex number $w(p)$ such that $L_p=w(p)+\afrak \afrak_M^{-1}$.
\begin{lem} \label{lemL0} For $p=0$ we have $L_0=\afrak \afrak_M^{-1}$.
\end{lem}
\begin{proof}
    Let $q \in L_0$. We can write $q=cx+dy$ where $(c,d) \in \afrak \times \afrak$ and $cm+dn=0$. Then
    \begin{align}
        nq & =ncx+ndy=ncx +y(-cm)=-c(my-nx)=-c \in \afrak, \nonumber \\
        mq & =mcx+mdy=(-dn)x +mdy=d \in \afrak.
    \end{align}
Hence $L_0 \subset \afrak(n)^{-1} \cap \afrak(m)^{-1}$. Conversely, suppose that $q \in \afrak(n)^{-1} \cap \afrak(m)^{-1}$. Then we can write $nq=c$ and $mq=d$ for some $(c,d)\in \afrak \times \afrak$. Then $q \in L_0$ since $  q=(xn-my)q=cx-dy$. Hence we have proved that $L_0=\afrak(n)^{-1} \cap \afrak(m)^{-1}$. In particular, we have $\afrak_M^{-1} \subset (n)^{-1}$ and $\afrak_M^{-1} \subset (m)^{-1}$ since $n \in \afrak_M$ and $m \in \afrak_M$ . It follows that $\afrak \afrak_M^{-1} \subset \afrak(n)^{-1} \cap \afrak(m)^{-1}=L_0$. On the other hand, we have that any $q=cx+dy \in L_0$ is in $\afrak \afrak_M^{-1}$ since $(x,y) \in \afrak_M^{-1} \times \afrak_M^{-1}$.
\end{proof}
Hence we have
\begin{align}
     (\Ehat^{(N)}_z(u,s),\Ehat^{(N)}_{\bar{z}}(u,s),\Ehat^{(N)}_v(u,s))= & \chi(\afrak)^{-1} \N(\afrak)^s\sum_{p \in \afrak \afrak_M} \sum_{q \in L_p} (
        1, 0, 0)J(\widetilde{\alpha},u)v(\widetilde{\alpha} u)^s, \nonumber \\
        & = \chi(\afrak)^{-1} \N(\afrak)^s\sum_{p \in \afrak \afrak_M} \sum_{q \in L_p} \eta(u,p,q,s)
    \end{align}
where $\widetilde{\alpha}=\alpha N= \begin{pmatrix} \ast & \ast \\ p & q \end{pmatrix}$. In particular, the $z$-component is
\begin{align}
     \Ehat_z^{(N)}(u,s)= \chi(\afrak)^{-1} \N(\afrak)^sv^s\sum_{p \in \afrak \afrak_M} \sum_{q \in L_p} \frac{\left ( \overline{p z+q}\right )^2}{\left (\vert p z+q \vert^2+\vert p v \vert^2 \right )^{2+s}}
\end{align}
where $p$ and $q$ are not both zero. The constant term is coming from the terms where $p=0$
\begin{align}
    \chi(\afrak)^{-1} \N(\afrak)^sv^s \sum_{q \in L_0} \frac{1}{ q^2\vert q \vert^{2s}}=\chi(\afrak)^{-1}\N(\afrak)^{s}v^sG(1+s,2,0,0;\afrak \afrak_M^{-1}),
\end{align}
since by Lemma \ref{lemL0} we have $L_0=\afrak \afrak_M^{-1}$. Hence at $s=0$ the constant term is $\chi(\afrak)^{-1}G_2(\afrak \afrak_M^{-1})$. As in \cite[p.~162]{ito} we find that the constant term of the $\bar{z}$-component at $s=0$ is
\begin{align}
    \restr{-\frac{\chi(\afrak)^{-1} \N(\afrak)^s2i\pi}{\Dcal(\afrak \afrak_M^{-1})}G(s,2,0,0;\afrak\afrak_M)}{s=0} & = -\frac{\chi(\afrak)^{-1} \Dcal(\afrak\afrak_M)}{\Dcal(\afrak\afrak_M^{-1})} G(\afrak\afrak_M) \nonumber \\
    & = -\chi(\afrak)^{-1} \N(\afrak_M)^2G_2(\afrak\afrak_M).
\end{align}
The last equality follows from the following lemma.
\begin{lem}
   For a fractional ideal $\bfrak$ we have $\Dcal(\bfrak)=-\sqrt{D}\N(\bfrak)$.
\end{lem}
\begin{proof}
 Recall that if $L=\omega_1\Z+\omega_2\Z$ with $\im(\omega_1/\omega_2)>0$ then
 \[\Dcal(L)=\omega_1\overline{\omega_2}-\overline{\omega_1}\omega_2=\begin{vmatrix} \omega_1 & \omega_2 \\ \overline{\omega_1} & \overline{\omega_2} \end{vmatrix}.\]
 Let $\bfrak=(a+b\tau)\Z+(c+d\tau)\Z$ a basis of $\bfrak$ such that
 \[\im \left (\frac{a+b\tau}{c+d\tau} \right )=\frac{(ad-bc)\sqrt{\vert D \vert}}{2\vert c+d\tau \vert}>0,\] 
 and where $\tau=\frac{1+\sqrt{D}}{2}$. Then
 \[\Dcal(\bfrak)=\begin{vmatrix} a+b\tau & c+d\tau \\[0.2cm] \overline{a+b\tau} & \overline{c+d\tau} \end{vmatrix}=\begin{vmatrix} a & b \\ c & d \end{vmatrix} \begin{vmatrix} 1 & 1 \\ \tau & \overline{\tau} \end{vmatrix}=-\sqrt{D}(ad-bc).\] Since the determinant $ad-bc$ is positive, we have
 \[ad-bc=\vert ad-bc \vert=\N(\bfrak).\]
\end{proof}

\begin{prop} \label{resteis}The restriction of the Eisenstein forms $\Ehat_{\chi,\afrak}$ to the boundary components are
    \[\res_r(\Ehat_{\chi,\afrak})= \chi(\afrak_c\afrak^{-1})\left (G_2(\afrak\afrak^{-1}_c)\omega_{\chi,r}- G_2(\afrak\bar{\afrak}_c^{-1}) \bar{\omega}_{\chi,r}\right ),\]
    where $c=[r] \in C_\Gamma$. For the Eisenstein series $E_{\chi,\afrak}$, the restriction is
    \[\res_r(E_{\chi,\afrak})=\delta_{\afrak,\afrak_c}\omega_{\chi,r}-\delta_{\afrak, \bar{\afrak}_c}\bar{\omega}_{\chi,r},\]
    where $\delta_{\ufrak,\vfrak}$ is defined by
    \begin{align*}
        \delta_{\ufrak,\vfrak} =  \begin{cases}
            1 & \textrm{if} \quad [\ufrak]=[\vfrak]  \\ 0 & \textrm{otherwise}.
        \end{cases}
    \end{align*}
    In particular, the restriction of $E_{\chi,\afrak}$ to the boundary is integral.
\end{prop}
\begin{proof} Let $M$ be any matrix in $\SL_2(K)$ sending  $r$ to $\8$, and $N=M^{-1}$. By \eqref{equivariance} we have
    \begin{align}
    \res_r(\Ehat_{\chi,\afrak})=M^\ast\res_\8(N^\ast \Ehat_{\chi,\afrak}).
    \end{align}
    where $ \res_{\8}(\omega)=\lim_{v \rightarrow \8} \iota_v^\ast \omega.$  After pulling back $N^\ast \Ehat_{\chi,\afrak}(u)$ by the map $\iota_v(z)=z+jv$ we get
    \begin{align}
      \iota_v^{\ast} N^\ast \Ehat_{\chi,\afrak}(u)=\Ehat^{(N)}_z(u)dz+\Ehat^{(N)}_{\bar{z}}(u)d\bar{z}.
      \end{align}
The limit of $\Ehat^{(N)}_z(u)$ as $v$ goes to $\8$ is the constant term in the Fourier expansion. Hence, we have
\begin{align}
    \lim_{v \rightarrow \8} \Ehat^{(N)}_z(u) & =\chi(\afrak)^{-1} G_2( \afrak\afrak_M^{-1}).
\end{align}
Similarly 
\begin{align} \label{limit1}
    \lim_{v \rightarrow \8} \Ehat^{(N)}_{\bar{z}}(u) & =-\chi(\afrak)^{-1} \N(\afrak_M)^2G_2(\afrak\afrak_M).
\end{align}
Since for any fractional $\cfrak$ ideal we have $\cfrak=\N(\cfrak)\bar{\cfrak}^{-1}$, by the homogeneity of $G_2$ we get
\begin{align}
    \lim_{v \rightarrow \8} \Ehat^{(N)}_{\bar{z}}(u) & =-\chi(\afrak)^{-1}G_2(\afrak \bar{\afrak}_M^{-1}).
\end{align}
and
\begin{align}
    \res_{\8} N^\ast \Ehat_{\chi,\afrak}=\chi(\afrak)^{-1}\left (G_2(\afrak\afrak^{-1}_M)dz- G_2(\afrak\bar{\afrak}_M^{-1}) d\bar{z}\right ).
\end{align}
After pulling back by $M$, and using that $[\afrak_M]=[\afrak_c]$ we get
\begin{align} \label{equationres}
    \res_{r} \Ehat_{\chi,\afrak}=\chi(\afrak_c\afrak^{-1})\left (G_2(\afrak\afrak^{-1}_c)\omega_{\chi,r}- G_2(\afrak\bar{\afrak}_c^{-1}) \bar{\omega}_{\chi,r}\right ).
\end{align}
Note that the right hand side of \eqref{equationres} does not depend on the ideal representatives $\afrak_c$. \\ \\
\indent Let $\res_r(E_{\chi,\afrak})=\alpha(\afrak)\omega_{\chi,r}+\beta(\afrak)\bar{\omega}_{\chi,r}$ be the restriction of $E_{\chi,\afrak}$ to the boundary component $r$. Since the restriction map is linear, it follows from Lemma \ref{eisensteinrelation} that
\begin{align}
    \res_r(\Ehat_{\chi,\afrak}) & = \sum_{i=1}^h \frac{\chi(\afrak_i\afrak^{-1})}{w(\afrak_i^{-1}\afrak)}G_2(\afrak_i^{-1}\afrak)\res_r(E_{\chi, \afrak_i}) \nonumber \\
    & = \left ( \sum_{i=1}^h \frac{\chi(\afrak_i\afrak^{-1})}{w(\afrak_i^{-1}\afrak)}G_2(\afrak_i^{-1}\afrak) \alpha(\afrak_i) \right )\omega_{\chi,r}+\left ( \sum_{i=1}^h \frac{\chi(\afrak_i\afrak^{-1})}{w(\afrak_i^{-1}\afrak)}G_2(\afrak_i^{-1}\afrak) \beta(\afrak_i)\right )\bar{\omega}_{\chi,r} .
\end{align}
Comparing with \eqref{equationres} we see that
\begin{align}
        \alpha(\afrak_i) =  \begin{cases}
            1 & \textrm{if} \quad [\afrak_i]=[\afrak_c]  \\ 0 & \textrm{otherwise}
        \end{cases}
    \end{align}
and
\begin{align}
        \beta(\afrak_i) =  \begin{cases}
            -1 & \textrm{if} \quad [\afrak_i]=[\bar{\afrak}_c]  \\ 0 & \textrm{otherwise}.
        \end{cases}
    \end{align}
Thus $\alpha(\afrak)=\delta_{\afrak,\afrak_c}$ and $\beta(\afrak)=-\delta_{\afrak,\bar{\afrak}_c}$.
\end{proof}
\begin{rmk}
    We will not discuss this further in the rest of the paper, but let us mention that the fact that the Eisenstein classes have an integral restriction to the boundary is an important property to find congruences. See \cite{berger_2009} for more on the relation between the denominator ideal, congruences and the Selmer group.
\end{rmk}

\begin{prop}
    The forms $\{ E_{\chi,\afrak}\}_{[\afrak] \in \Cl(K)}$ span the Eisenstein cohomology $H^1_{\Eis}(Y_\Gamma;\C)$.
\end{prop}
\begin{proof}
Let $V$ be the subspace of the Eisenstein cohomology spanned by the forms $E_{\chi,\afrak}$. We know that the Eisenstein cohomology is $h$-dimensional, so $\dim(V)\leq h$. A class in $H^1(\partial X_\Gamma;\C)$ is represented by a collection of forms $\alpha_r\omega_{\chi,r}+\beta_r\bar{\omega}_{\chi,r}$ with $\alpha_r,\beta_r$ two complex numbers such that $\alpha_{\gamma r}=\alpha_r$ and $\beta_{\gamma r}=\beta_r$. By Proposition \ref{imres}, the image of the restriction is the $-1$ eigenspace by the involution $\iota$. Hence if the class lies in $\im(\Res)$, then we have in addition $\beta_r=-\alpha_{\bar{r}}$, by \eqref{involution2}. It follows that the space $\im(\Res)$ is spanned by the forms $\omega_{\chi,r}-\bar{\omega}_{\chi,\bar{r}}$. By the previous proposition, we have
\begin{align}
    \Res(E_{\chi,\afrak_c})=\omega_{\chi,r}-\bar{\omega}_{\chi,\bar{r}},
\end{align}
Hence the restriction map $\Res \colon V \longrightarrow \im(\Res)$ is surjective and $\dim(V)=h$.
\end{proof}

\subsection{Relation to the Sczech cocycle} \label{Sczechrel}

The forms $\Ehat_{\chi,\afrak}$ in $\Omega^1(Y_\Gamma;\C)$ define a cocycle in $H^1(\Gamma;\C)$ by
\begin{align}
\Ehat_{\chi,\afrak}(\gamma) \coloneqq \int_{u_0}^{\gamma u_0} \Ehat_{\chi,\afrak}.
\end{align}
Since the form is closed, the integral does not depend on the path from $u_0$ to $\gamma u_0$. The following result shows that this cocycle is Sczech's cocycle.

\begin{thm} \label{eisensteinszcech} We have\begin{align}
   \Ehat_{\chi,\afrak}\begin{pmatrix}
    a & b \\ c & d
    \end{pmatrix} = \chi(\afrak)^{-1}
    \begin{cases}
    I \left ( \frac{a+d}{c} \right )G_2(\afrak)-D(a,c,\afrak) & \textrm{if} \quad c \neq 0, \nonumber \\[0.3cm]
   I \left ( \frac{b}{d} \right )G_2(\afrak) & \textrm{if} \quad c = 0
    \end{cases}
\end{align} In particular, we have $\chi(\afrak)\Ehat_{\chi,\afrak}(\gamma)=\Phi_\afrak(\gamma)$.
\end{thm}

\begin{proof}
Let \[\gamma=\begin{pmatrix}
a & b \\ c & d
\end{pmatrix} \in \Gamma.\]
\begin{enumerate} [wide, labelindent=0pt]
\item Suppose that $c \neq 0$. For a real number $\epsilon>0$ consider the closed path $P_\epsilon(\gamma)$ in $\HH_3$ pictured in Figure \ref{movingtoboundarya}. Let us denote by $[u_i,u_j]$ the oriented segment from $u_i$ to $u_j$.\begin{figure}[h!]
\centering
\captionsetup{width=.75\textwidth,font={normalsize,it}}
\includegraphics[scale=0.5]{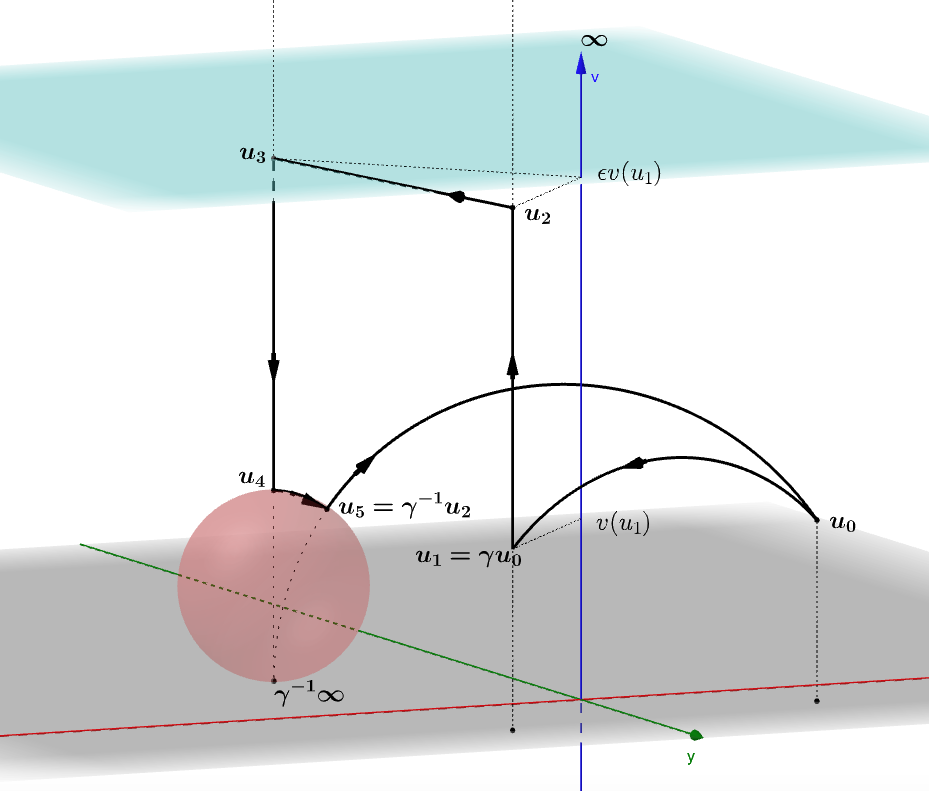}
\caption{The path $P_\epsilon(\gamma)$ when $c\neq0$. The blue axis is the $v$-axis of $u=z+jv$, whereas the green (resp. red) axis is the $x$-axis ( resp $y$-axis). We start with the geodesic segment joining $u_0$ to $u_1=\gamma u_0$. We then consider the geodesic segment from $u_1=z(v_1)+jv(u_1)$ to $u_2=z(v_1)+j\epsilon v(u_1)$ for some $\epsilon >0$, along the vertical geodesic line through $u_1$. We denote by $\Hcal_{\8,\epsilon v(u_1)}$ the blue horosphere parallel to the $xy$-axis and containing $u_2$. We consider the cusp $\gamma^{-1}\8=-\frac{d}{c}$. The red horosphere tangent to $\gamma^{-1}\8$ is $\gamma^{-1}\Hcal_{\8,\epsilon v(u_1)}$.The points $u_3$ and $u_4$ are the intersection points between the vertical geodesics and the two horospheres. Note that $\gamma[u_5,u_0]=-[u_1,u_2]$.}
\label{movingtoboundarya}
\end{figure}
\begin{figure}[h!]
\centering
\includegraphics[scale=0.4]{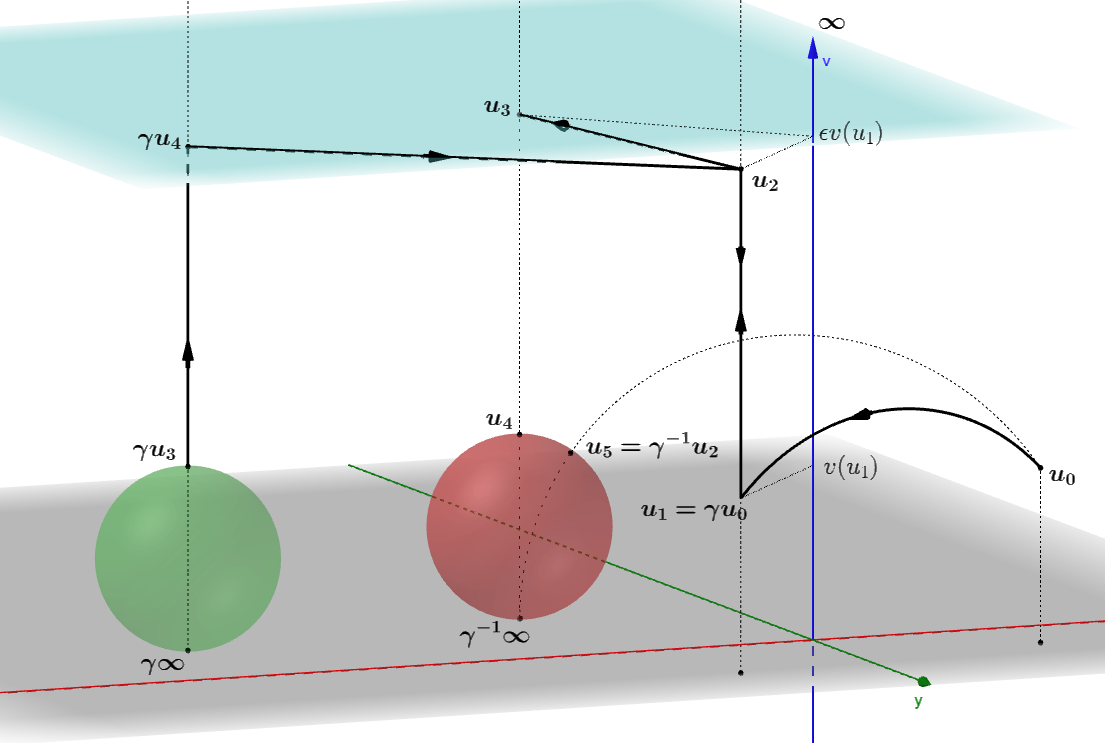}
\captionsetup{width=.75\textwidth,font={normalsize,it}}
\caption{The path $\widetilde{P}_\epsilon(\gamma)$ is $\Gamma$-equivalent to the path $P_\epsilon(\gamma)$ in Figure \ref{movingtoboundarya}. Note that $\gamma[u_5,u_0]=-[u_1,u_2]$. Hence we can translate $[u_5,u_0]$ by $\gamma$ and get twice the same geodesic segment between $u_1$ and $u_2$ with opposite orientations. Then, we translate by $\gamma$ the geodesic between $\8$ and $\gamma^{-1}\8$ containing $u_3$ and $u_4$ to the geodesic between $\gamma \8$ and $\8$. The segment $[u_3,u_4]$ is translated to $[\gamma u_3,\gamma u_4]$. The segment $[u_4,u_5]$ contained in the red horosphere tangent to $\gamma^{-1}\8$ is sent to the segment $[\gamma u_4,u_2]$ contained in the blue horosphere containing $u_2$ and $u_3$.}
\label{movingtoboundaryb}
\end{figure}Since the path and the form are closed, by Stoke's theorem we have
\begin{align}
    \int_{P_\epsilon(\gamma)} \Ehat_{\chi,\afrak}=0
\end{align}
for every $\epsilon>0$, and we will take the limit $\epsilon \rightarrow \8$. The path $P_\epsilon(\gamma)$ can be translated by $\Gamma$ to the path $\widetilde{P}_\epsilon(\gamma)$ in Figure \ref{movingtoboundaryb}. Since the form is $\Gamma$-invariant, we have
\begin{align}
    \int_{\widetilde{P}_\epsilon(\gamma)}\Ehat_{\chi,\afrak}=\int_{P_\epsilon(\gamma)} \Ehat_{\chi,\afrak}=0.
\end{align}
Moreover, since the two integrals along $[u_1,u_2]$ and $[u_2,u_1]$ cancel, it follows from Figure \ref{movingtoboundaryb} that 
\begin{align} 
    0 & = \int_{u_0}^{\gamma u_0} \Ehat_{\chi,\afrak} + \int_{\gamma u_3}^{\gamma u_4} \Ehat_{\chi,\afrak}  + \int_{\gamma u_4}^{u_2} \Ehat_{\chi,\afrak} + \int_{u_2}^{u_3} \Ehat_{\chi,\afrak} \nonumber \\
    & = \int_{u_0}^{\gamma u_0} \Ehat_{\chi,\afrak}+ \int_{\gamma u_3}^{\gamma u_4} \Ehat_{\chi,\afrak}  + \int_{\gamma u_4}^{u_3} \Ehat_{\chi,\afrak}.
\end{align}
Note that the endpoints $u_3,u_4,\gamma u_3$ and $\gamma u_4$ depend on $\epsilon$, and that the equality above holds for any $\epsilon >0$. Taking the limit we get
\begin{align} \label{twointegrals}
    \int_{u_0}^{\gamma u_0} \Ehat_{\chi,\afrak}= -\lim_{\epsilon \rightarrow \8} \int_{\gamma u_3}^{\gamma u_4} \Ehat_{\chi,\afrak}  -\lim_{\epsilon \rightarrow \8} \int_{\gamma u_4}^{u_3} \Ehat_{\chi,\afrak}
\end{align}We will compute these two integrals separately.

\begin{enumerate}[label=\theenumi\alph*), wide, labelindent=0pt]
    \item[$\bullet$] We begin with the integrals inside the blue horosphere at the cusp $\8$. By definition of the restriction $\Ehat_{\chi,\afrak}^{(\8)}$ of $\Ehat_{\chi,\afrak}$ to the boundary at $\8$, we have
    \begin{align} \label{limint}
        \lim_{\epsilon \rightarrow \8} \int_{\gamma u_4}^{u_3}\Ehat_{\chi,\afrak} = \int_{z}\Ehat_{\chi,\afrak}^{(\8)}
    \end{align}
    where $z$ is the image in the horosphere $\Hcal_\8 \simeq \C$ at $\8$ of the segment $[\gamma u_4,u_3]$. Moreover, by Proposition \ref{resteis} the restriction at the cusp $c=\8$ is  \begin{align}
    \Ehat_{\chi,\afrak}^{(\8)}=\chi(\afrak^{-1})G_2(\afrak)(dz-d\bar{z}).    
    \end{align}Let us parametrise the segment $z$ more precisely. We have $\gamma \8=\frac{a}{c}$ and $\gamma^{-1} \8=-\frac{d}{c}$. Hence we have $\gamma u_4=\frac{a}{c}+\epsilon v(u_1)$ and $u_3=-\frac{d}{c}+\epsilon v(u_1)$. Hence, inside the horosphere, the cycle $z$ is the segment $z(t)=\frac{a}{c}-t\frac{a+d}{c}$ with $0 \leq t \leq 1$. Thus, the integral is
\begin{align} \label{integral1}
   \int_{z}\Ehat^{(\8)}_{\chi,\afrak} & =\chi(\afrak)^{-1}\int_0^1 \left (G_2(\afrak)  z(t)-G_2(\afrak)\bar{z}(t) \right )dt \nonumber \\
    & = -\chi(\afrak)^{-1} \left ( G_2(\afrak) \frac{a+d}{c} -G_2(\afrak) \overline{\left ( \frac{a+d}{c}\right )} \right ) \nonumber \\
    & = -\chi(\afrak)^{-1}  G_2(\afrak) I\left (\frac{a+d}{c} \right ).
\end{align}
\item[$\bullet$] Let us now compute the integral along $[\gamma u_3,\gamma u_4]$. The path $[\gamma u_3,\gamma u_4]$ in Figure \ref{movingtoboundaryb} is parametrized by $u(t)=\frac{a}{c}+jt$ with $d(\epsilon)<t<\epsilon v(u_1)$, where $d(\epsilon)=v(\gamma u_3)$ is the diameter of the green horosphere tangent at $\gamma \8$. Using \eqref{gammav}, one can see that
$d(\epsilon)=\frac{1}{\vert c \vert^2 \epsilon v(u_1)}$. In particular, we have $d(\epsilon) \rightarrow 0$ as $\epsilon$. We can compute the integral 
\begin{align}
    \int_{\gamma u_3}^{\gamma u_4} \Ehat_{\chi,\afrak}(u,s) & =  \int_{\frac{1}{\vert c \vert^2 \epsilon v(u_1)}}^{\epsilon v(u_1)} \Ehat_{\chi,\afrak,v}(u,s) \nonumber \\
    & = 2\chi(\afrak)^{-1}\N(\afrak)^s \sum_{(m,n)\in \afrak \times \afrak} \int_{\frac{1}{\vert c \vert^2 \epsilon v(u_1)}}^{\epsilon v(u_1)} \frac{\left ( \overline{m\frac{a}{c}+n}\right )\overline{m}t^{1+s}}{\left (\vert m\frac{a}{c}+n \vert^2+\vert m \vert^2 t^2 \right )^{2+s}}dt \nonumber \\
    & = 2\chi(\afrak)^{-1}\N(\afrak)^s \sum_{(m,n)\in \afrak \times \afrak} \frac{\left ( \overline{m\frac{a}{c}+n}\right )\overline{m}}{\vert m\frac{a}{c}+n \vert^{2(2+s)}} \int_{\frac{1}{\vert c \vert^2 \epsilon v(u_1)}}^{\epsilon v(u_1)} \frac{t^{1+s}}{\left (1+\frac{\vert m \vert^2}{\vert m\frac{a}{c}+n \vert^2} t^2 \right )^{2+s}}dt.
\end{align}
By substituting $\alpha=\frac{\vert m \vert}{\vert m\frac{a}{c}+n \vert} t$ and taking the limit as $\epsilon$ goes to $\8$, this becomes
\begin{align} \label{sum3142}
    & 2\chi(\afrak)^{-1}\N(\afrak)^s \sum_{(m,n)\in \afrak \times \afrak} \frac{\left ( \overline{m\frac{a}{c}+n}\right )}{\vert m\frac{a}{c}+n \vert^{2+s}} \frac{\overline{m}}{\vert m \vert^{2+s}} \int_0^\8 \frac{\alpha^{1+s}}{\left (1+\alpha^2 \right )^{2+s}}d\alpha \nonumber \\
    & = \chi(\afrak)^{-1}\N(\afrak)^s B\left (1+\frac{s}{2},1+\frac{s}{2} \right ) \sum_{(m,n)\in \afrak \times \afrak} \frac{\left ( \overline{m\frac{a}{c}+n}\right )}{\vert m\frac{a}{c}+n \vert^{2+s}} \frac{\overline{m}}{\vert m \vert^{2+s}},
\end{align}
where
\begin{align}
     B\left (x,y \right )=\int_0^\8 \frac{t^{y-1}}{(1+t)^{x+y}} dt = \frac{\Gamma \left (x \right )\Gamma \left (y \right )}{\Gamma(x+y)}
\end{align}
is the Beta function. By writing $m=c\widetilde{m}+r$ and summing over $\widetilde{m}$ and $r$ we rewrite the inner sum in \eqref{sum3142} as
\begin{align}
     \sum_{r \in \afrak/c\afrak}\sum_{(\widetilde{m},n)\in \afrak \times \afrak} \frac{\left ( \overline{r\frac{a}{c}+a\widetilde{m}+n}\right )}{\vert r\frac{a}{c}+a\widetilde{m}+n \vert^{2+s}} \frac{\overline{c\widetilde{m}+r}}{\vert c\widetilde{m}+r \vert^{2+s}},
\end{align}
where the inner sum is restricted to $(\widetilde{m}, n)\neq (-r/c,0)$. By summing over $(\widetilde{m},a\widetilde{m}+n)$ instead of $(\widetilde{m},n)$, we rewrite the sum as
\begin{align}
   \sum_{r \in \afrak/c\afrak} \frac{1}{c\vert c \vert^s} G \left ( \frac{1+s}{2},1,\frac{ar}{c},0,\afrak \right )G \left ( \frac{1+s}{2},1,\frac{r}{c},0,\afrak \right ),
\end{align}
where the inner sum in the first line is restricted to $(\widetilde{m}, \widetilde{n})\neq (-r/c,-ar/c)$. Hence at $s=0$ we get
\begin{align}
    \int_{\gamma u_3}^{\gamma u_4} \Ehat_{\chi,\afrak} = \chi(\afrak)^{-1}B(1,1)\frac{1}{c} \sum_{r \in \afrak/c\afrak}  G_1 \left ( \frac{ar}{c},\afrak \right )G_1 \left (\frac{r}{c},\afrak \right ) =\chi(\afrak)^{-1}D(a,c,\afrak).
\end{align}
It follows from \eqref{twointegrals} and \eqref{integral1} that when $c \neq 0$:
\begin{align}
     \int_{u_0}^{\gamma u_0} \Ehat_{\chi,\afrak}=  \chi(\afrak)^{-1} \left ( G_2(\afrak) I\left (\frac{a+d}{c} \right )-D(a,c,\afrak) \right ).
\end{align}
\end{enumerate}

\item Suppose that $c=0$. Let $P_\epsilon(\gamma)$ the closed path pictured in Figure \ref{movingtoboundaryc}. \begin{figure}[h!]
\centering
\includegraphics[scale=0.45]{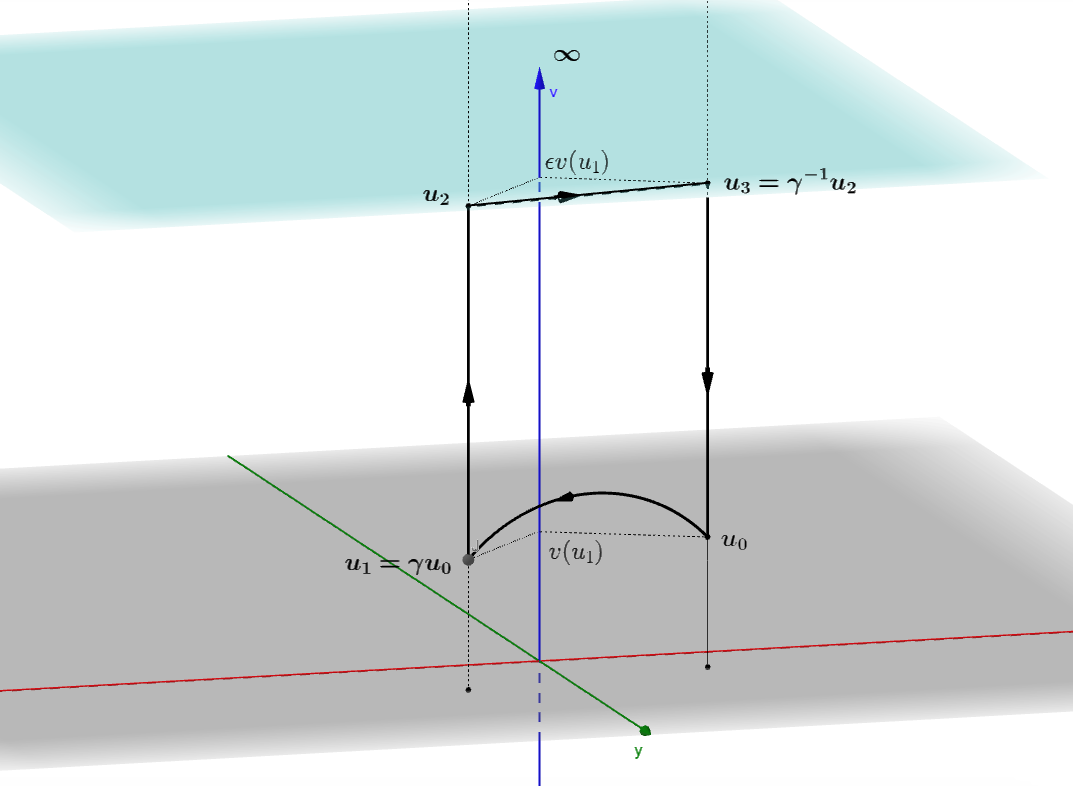}
\captionsetup{width=.75\textwidth,font={normalsize,it}}
\caption{The path $P_\epsilon(\gamma)$ for $c=0$. The two sides of the square are $\gamma$ translate of each other with opposite directions.}
\label{movingtoboundaryc}
\end{figure} The integrals along the paths $[u_1,u_2]$ and $[u_3,u_0]$ in Figure \ref{movingtoboundaryc} cancel since $[u_1,u_2]=-\gamma[u_3,u_0]$ and $\Ehat_{\chi,\afrak}$ is $\Gamma$-invariant. As in \eqref{limint} we have
\begin{align}
\int_{u_0}^{\gamma u_0} \Ehat_{\chi,\afrak} = - \lim_{\epsilon \rightarrow \8} \int_{u_2}^{u_3} \Ehat_{\chi,\afrak} = \int_z \Ehat_{\chi,\afrak}^{(\8)}
\end{align}
where $z$ is the segment in $\Hcal_\8$ obtained by moving $[u_2,u_3]$ to $\8$. Since $K$ has no non-trivial units, we have $a,d=\pm 1$ and $a/d=1$. If $u_0=z_0+jv_0$, then $\gamma u_0=z_0+\frac{b}{d} + jv_0$. Hence, the path $z$ is parametrized by 
$z(t)=z_0+\frac{b}{d}(1-t)$ for $0 \leq t \leq 1$. At the boundary the form is
\begin{align}
    \Ehat_{\chi,\afrak}^{(\8)}=\chi(\afrak)^{-1}G_2(\afrak)\left ( dz-d\bar{z} \right )
\end{align}
and thus
\begin{align}
   \int_{z} \Ehat_{\chi,\afrak}^{(\8)} = \chi(\afrak)^{-1}G_2(\afrak) \int_0^1 (z(t)-\bar{z}(t))dt= -\chi(\afrak)^{-1} G_2(\afrak) I \left (\frac{b}{d} \right ).
\end{align}
\end{enumerate}
\end{proof}
\begin{cor} The forms $2\Ehat^{\alg}_{\chi,\afrak} \coloneqq 2\Omega^{-2} \Ehat_{\chi,\afrak}$ are in $\cohomt^1(X_\Gamma,{\Ocal_F})$.
\end{cor}
\begin{proof}
    We have
    \begin{align} \label{anequalityscz}
        \int_{u_0}^{\gamma u_0} 2\Ehat^{\alg}_{\chi,\afrak}  = 2\Omega^{-2}  \int_{u_0}^{\gamma u_0} \Ehat_{\chi,\afrak} = \Omega^{-2} \chi(\afrak^{-1}) 2\Phi_{\afrak}(\gamma) = \chi(\afrak^{-1})\lambda(\afrak^{-1})^2 2\Phi_{L_{\afrak^{-1}}}(\gamma).
    \end{align}
By Proposition \eqref{proplambda} the factor $\chi(\afrak)\lambda(\afrak)^2$ is a unit in ${\Ocal_F}$ for any fractional ideal $\afrak$, and the result follows from the integrality of the Sczech cocycle.
\end{proof}

\section{Denominators of the Eisenstein cohomology}

We have seen that $H_{\Eis}^1(X_\Gamma;\C)$ is an $h$-dimensional complex vector space spanned by the forms $E_{\chi,\afrak}$. The ${\Ocal_F}$-lattice $\Lcal_{\Eis}$ of $H^1_{\Eis}(X_\Gamma,\C)$ defined by
\begin{align}
    \Lcal_{\Eis}\coloneqq \Eis(\cohomt^1(\partial X_\Gamma;\Ocal_F))=\bigoplus_{[\afrak] \in \Cl(K)}E_{\chi,\afrak} {\Ocal_F}
\end{align}
gives us an integral structure on $H^1_{\Eis}(X_\Gamma,\C)$, {\em i.e} $\Lcal_{\Eis}\otimes_{\Ocal_{F}} \C \simeq H^1_{\Eis}(X_\Gamma,\C)$. Another ${\Ocal_F}$-lattice is given by the integral Eisenstein classes
\begin{align}
    \Lcal_{0}\coloneqq \cohomt^1(X_\Gamma;{\Ocal_F}) \cap H_{\Eis}^1(X_\Gamma;\C) \subset H_{\Eis}^1(X_\Gamma;\C).
\end{align}
We define the denominator of the Eisenstein cohomology to be the $\Ocal_F$-ideal
\begin{align}
    \Den(\Lcal_{\Eis}) \coloneqq \left \{ \left . \lambda \in \Ocal_F \right \vert \lambda \Lcal_{\Eis} \subset \Lcal_0 \right \}.
\end{align}

\begin{lem}
    The submodule $\Lcal_{\Eis}$ does not depend on the choice of the Hecke character $\chi$. 
\end{lem}
\begin{proof}
    Let us temporarily denote by $\Lcal_{\chi,\Eis}$ and $\Lcal_{\widetilde{\chi},\Eis}$ the lattice that we obtain for two different Hecke characters $\widetilde{\chi}$ and $\chi$ of infinity type $(-2,0)$. Then $\varphi=\chi/\widetilde{\chi}$ is an $F$-valued character on the class group of $K$. Since $E_{\widetilde{\chi},\afrak}=\varphi(\afrak)^{-1}E_{\chi,\afrak}$, we have $\Lcal_{\chi,\Eis}=M\Lcal_{\widetilde{\chi},\Eis}$, where
    \[M=\operatorname{diag}(\varphi(\afrak_1)^{-1},\dots,\varphi(\afrak_h)^{-1}).\]
    Since $\varphi(\afrak)^h=1$, the value $\varphi(\afrak)$ is a unit in $\Ocal_{F}$ and $M \in \Mat_h(\Ocal_{F})$. Furthermore, the determinant of $M$ is
    \[\det(M)=\prod_{\afrak \in \Cl(K)} \varphi(\afrak)^{-1}.\]
    The substitution $\afrak \rightarrow \afrak^{-1}$ shows that $\det(M)=\det(M)^{-1}$. Hence $\det(M)=\pm 1$ and $M$ is in $\GL_h(\Ocal_{F})$.
\end{proof}

\subsection{The upper bound}

We will need the following result on Dedekind determinants, of which a proof can be found in \cite[Chapter.~3, Theorem.~6.1]{langcyclo}.
\begin{lem} \label{langlem}
    Let $f$ any complex valued function on a finite abelian group $G$. Then
    \[\det (f(a^{-1}b))_{a,b \in G}=\prod_{\varphi \in \widehat{G}}\left ( \sum_{a \in G}\varphi(a)f(a^{-1}) \right ).\]
\end{lem}

\begin{thm} \label{mainthm} The form
\[\frac{1}{2\sqrt{D}} L^{\integ}( \chi \circ \N_{H/K},0)E_{\chi,\afrak}\]
is integral in $\Ocal_F$ . In particular, the denominator has the upper bound
\[L^{\integ}( \chi \circ \N_{H/K},0){\Ocal_F} \subset  \Den(\Lcal_{\Eis}).\]
\end{thm}

\begin{proof} 
Let $M_{\chi} \in \GL_h(F)$ be the matrix such that
\begin{align} \label{equality}
    \begin{pmatrix}
        \Ehat^{\alg}_{\chi,\afrak_1} \\      
        \vdots \\
        \Ehat^{\alg}_{\chi,\afrak_h}
    \end{pmatrix}=M_{\chi} \begin{pmatrix}
        E_{\chi,\afrak_1} \\      
        \vdots \\
        E_{\chi,\afrak_h}
    \end{pmatrix}.
\end{align}
By Proposition \ref{continuation}, we have
\[M_\chi = \left [ \Omega^{-2}\frac{\chi(\afrak\bfrak^{-1})}{w(\afrak^{-1}\bfrak)}G_2(\afrak^{-1}\bfrak) \right ]_{[\afrak],[\bfrak] \in \Cl(K)}.\]
Applying the previous lemma to the function 
\[f(\afrak) =\Omega^{-2} \frac{\chi(\afrak^{-1})}{w(\afrak)}G_2(\afrak), \] we then have
\begin{align}
    \det(M_\chi) & = \prod_{\varphi \in \widehat{\Cl(K)}}\left ( \Omega^{-2} \sum_{\afrak \in \Cl(K)}\varphi(\afrak)\frac{\chi(\afrak)}{w(\afrak^{-1})}G_2(\afrak^{-1}) \right ) \nonumber \\
    & = \prod_{\varphi \in \widehat{\Cl(K)}}\Omega^{-2} L(\varphi \chi,0) \nonumber \\
    & = \prod_{\varphi \in \widehat{\Cl(K)}}L^{\alg}(\varphi \chi,0) \nonumber \\
    & = L^{\alg}( \chi \circ \N_{H/K},0).
\end{align}
We have already seen in the proof of Proposition \ref{Lalgint} that for any fractional ideal $\afrak$
\begin{align}
    4\sqrt{D}f(\afrak^{-1}) & = 4\sqrt{D} \frac{\chi(\afrak)}{w(\afrak^{-1})}\lambda(\afrak)^2G_2(L_\afrak) \in {\Ocal_F}
\end{align}
is integral. Hence the matrix $\widetilde{M}_{\chi}\coloneqq 4\sqrt{D}M_\chi$ is in $\GL_h(F) \cap \Mat_h({\Ocal_F})$. Furthermore, the adjoint matrix $N_\chi\coloneqq \det(\widetilde{M}_\chi)\widetilde{M}^{-1}_\chi$ also has coefficients in ${\Ocal_F}$ and 
\begin{align}
    N_{\chi}M_{\chi} & =4^{-1}D^{-\frac{1}{2}}N_\chi \widetilde{M}_\chi \nonumber \\
    & =4^{-1}D^{-\frac{1}{2}}\det(\widetilde{M}_\chi) \nonumber \\
    & =4^{h-1}D^{\frac{h-1}{2}}\det(M_\chi) \nonumber \\
    & =4^{h-1}D^{\frac{h-1}{2}}L^{\alg}( \chi \circ \N_{H/K},0) \nonumber \\
    & =4^{-1}D^{-\frac{1}{2}}L^{\integ}( \chi \circ \N_{H/K},0).
\end{align} Applying $2N_\chi$ to both sides of \eqref{equality} we find that
\begin{align}
    2N_\chi\begin{pmatrix}
        \Ehat^{\alg}_{\chi,\afrak_1} \\      
        \vdots \\
        \Ehat^{\alg}_{\chi,\afrak_h}
    \end{pmatrix}=\frac{1}{2\sqrt{D}}L^{\integ}( \chi \circ \N_{H/K},0) \begin{pmatrix}
        E_{\chi,\afrak_1} \\      
        \vdots \\
        E_{\chi,\afrak_h}
    \end{pmatrix}.
\end{align}
Since the left hand side is integral, it follows that
\begin{align}
    \frac{1}{2\sqrt{D}}L^{\integ}( \chi \circ \N_{H/K},0)E_{\chi,\afrak_i}
\end{align} is integral for any $i$.
\end{proof}

\subsection{Relation to the work of Berger} \label{bergerrel}
 In his work \cite{berger1}, Berger uses the adelic setting that we will need to convert to the classical setting. Let us suppose for the rest of this section that $K$ has class number one. A more general result relating the adelic Eisenstein cohomology to Sczech's cocycle is due to Weselmann, see in particular \cite[Bemerkung.~2, p.~116]{weselmann}.

We recall the setting of \cite[Section.~4.3]{berger_2009}. Let $\phi=(\phi_1,\phi_2)$ be a pair of Hecke characters on the idèles $\A_K^\times$, where $\phi_{1,\8}(z)=z$ and $\phi_{2,\8}(z)=z^{-1}$. Hence, the character $\chi=\phi_1/\phi_2$ has infinity type $\chi_\8(z)=z^2$. Suppose that $\chi$ is unramified.
Let $G \coloneqq \GL_{2/K}$, let $K_\8\coloneqq \C^\ast \cdot \U(2)$. Let $\Mfrak_i$ be the conductor of $\phi_i$ and set
\begin{align}
K_\fin = \prod_{v \mid \Mfrak_1} U^1(\Mfrak_{1,v}) \prod_{v \mid \Mfrak_1} \GL_2(\Ocal_v)
\end{align}
where $U^1(\Mfrak_{1,v})= \left \{ \left . k \in \GL_2(\Ocal_v) \right \vert  \det(k) \equiv 1 \mod \Mfrak_{1,v}\right \}$ and $\Mfrak_{1,v}$ is the ideal $\Mfrak_{1}\Ocal_v$. We define the adelic space
\begin{align}
    S_{K_\fin} \coloneqq G(\Q) \backslash G(\A)/K_\8K_\fin.
\end{align}
Since the class number of $K$ is one, we have a decomposition
\begin{align} \label{decompadel}
    G(\A_\fin)=G(\Q)K_\fin.
\end{align} This yields the isomorphism
\begin{align} \label{isomoradelic}
    S_{K_\fin} & \longrightarrow Y_{\Gamma} \nonumber \\
    G(\Q) (g_\8,g_\fin)K_\8K_\fin & \longmapsto \Gamma g^{-1}g_\8 j,
\end{align}
where we write $g_\fin=gk_\fin$ according to the decomposition \eqref{decompadel}, and $\Gamma= G(\Q) \cap K_\fin $.

\begin{rmk} The fact that $\Gamma= G(\Q) \cap K_\fin$ is only true if $1$ is the only unit of $\Ocal^\ast$ that is congruent to $1$ modulo $\Mfrak$. If $-1$ is congruent $1$ modulo $\Mfrak$, then $G(\Q) \cap K_\fin = \GL_2(\Ocal)$. However this does not make a difference for the rest of the discussion. The Eisenstein series on $\GL_2(\Ocal) \backslash \HH_3$ and $\SL_2(\Ocal) \backslash \HH_3$ are the same, up to a potential factor $2$. However this factor will be irrelevant since we will consider the denominator away from certain primes dividing $2D$.
\end{rmk}

Let $\g$ be the Lie algebra of $G(\R)$ and $\kfrak$ be the Lie algebra of $K_\8$. We have an orthogonal decomposition $\g=\kfrak \oplus \p$ with respect to the Killing form. We consider the following elements in $\p_\C$:
\begin{align}
X & = \begin{pmatrix}
        0 & 1 \\ 1 & 0
    \end{pmatrix} , \qquad  Y=\begin{pmatrix}
        0 & i \\ -i & 0
    \end{pmatrix}, \qquad S =\frac{1}{2} \left ( X \otimes_{\R} 1 - Y \otimes_{\R} i \right ).
\end{align}
\begin{lem}
After identifying $\p$ with the tangent space of $\HH_3$ at $j$, we have
\[X \simeq 2\restr{\frac{\partial}{\partial x}}{u=j}, \qquad Y \simeq 2\restr{\frac{\partial}{\partial y}}{u=j}, \qquad S \simeq 2\restr{\frac{\partial}{\partial z}}{u=j}= \restr{\frac{\partial}{\partial x}}{u=j}-i\restr{\frac{\partial}{\partial y}}{u=j}.\]
\end{lem}
\begin{proof}
    Let us prove the statement for $X$. Let $f(u)$ be a function on $\HH_3$ with coordinates $u=x+iy+jv$. The action of $X$ on $f$ at $u=j$ is given by
    \begin{align}
        (Xf)(j) \coloneqq \restr{\frac{d}{dt}}{t=0}f(\Phi_X(t))
    \end{align}
    where 
    \begin{align}
        \Phi_X(t) \coloneqq \exp(tX)j=\begin{pmatrix}
            \cosh(t) & \sinh(t) \\
            \sinh(t) & \cosh(t)
        \end{pmatrix}j.
    \end{align}
In the coordinates $u=x+iy+jv$, the map is explicitely given by
\begin{align}
    \Phi_X(t)= \frac{1}{\cosh(t)^2+\sinh(t)^2}(\sinh(2t),0,1) \in \HH_3 \simeq \R^2 \times \R_{>0}
\end{align} and its derivative is 
\begin{align}
    \Phi'_X(0)=(2,0,0).
\end{align} Hence
\begin{align}
    (Xf)(j)= \Phi'_X(0) d_{u=j}f= 2\frac{\partial f}{\partial x}(j).
\end{align} The computation for $Y$ is similar.

\end{proof}
Consider the Borel subgroup 
\begin{align}
    B(\Q) \coloneqq \left \{ \begin{pmatrix}
        a & b \\ 0 & d
    \end{pmatrix} \in \GL_2(K)\right \}.
\end{align} We have $G(\Q)=B(\Q)\Gamma$, hence \eqref{decompadel} becomes 
\begin{align} \label{decompadel2}
    G(\A_\fin)=B(\Q)K_\fin
\end{align}since $\Gamma \subset K_\fin$. We have an isomorphism similar to \eqref{isomoradelic}:
\begin{align} \label{isomoradelic2}
    B(\Q) \backslash G(\A)/K_\8K_\fin & \longrightarrow \Gamma_\8 \backslash \HH_3 \nonumber \\
    B(\Q) (g_\8,g_\fin)K_\8K_\fin & \longmapsto \Gamma_\8 b_\8^{-1}g_\8 j,
\end{align}
where $g_\8=b_\8k_\fin$ according to the decomposition \eqref{decompadel2}.
We have an isomorphism
\begin{align} \label{isoform1}
     \Hom_{K_\8}(\p, C^{\8} \left (G(\Q) \backslash G(\A)/K_\fin \right )) \longrightarrow \Omega^1(\HH_3)^{\Gamma},
\end{align}
where $K_\8$ acts by the adjoint representation on $\p$. Hence, an element $\omegatilde[X](g_\8,g_\fin)$ in the space on the left-hand side satisfies $\omegatilde[X](g_\8k_\8,g_\fin)=\omegatilde[\Ad(k_\8)^{-1}X](g_\8,g_\fin)$. The map goes as follows.  Let $u=g_\8j$ be a point in $\HH_3$ and $X_u \in T_u\HH_3$ a tangent vector at $u$. We define a differential form in $\Omega^1(\HH_3)^{\Gamma}$ by
\begin{align} \label{isoadel1}
    \omega_u(X_u) \coloneqq \omegatilde[d_uL_{g_\8^{-1}}X_u](g_\8,1)
\end{align}
where $L_{g}(u)=gu$ is the map induced by the action of $G(\R)$ on $\HH_3$. Similarly, we have
\begin{align} \label{isoform2}
     \Hom_{K_\8}(\p, C^{\8} \left (B(\Q) \backslash G(\A)/K_\fin \right )) \longrightarrow \Omega^1(\HH_3)^{\Gamma_\8}.
\end{align}
The Eisenstein operator considered by Berger is
\begin{align} \label{operatoreis}
   \Eis \colon \Hom_{K_\8}(\p, C^{\8} \left (B(\Q) \backslash G(\A)/K_\fin \right )) & \longrightarrow  \Hom_{K_\8}(\p, C^{\8} \left (G(\Q) \backslash G(\A)/K_\fin \right )) \nonumber \\
   \omegatilde & \longrightarrow \Eis(\omegatilde)[X](g) \coloneqq \sum_{\gamma \in B(\Q)\backslash G(\Q)} \omegatilde[X](\gamma g).
\end{align}
It is immediate from the definitions that the following diagram is commutative
\begin{equation} \label{adiagram}
\begin{tikzcd}
  \Hom_{K_\8}(\p, C^{\8} \left (B(\Q) \backslash G(\A)/K_\fin \right )) \arrow[d] \arrow[r,"\Eis"] &  \Hom_{K_\8}(\p, C^{\8} \left (G(\Q) \backslash G(\A)/K_\fin \right )) \arrow[d] \\
 \Omega^1(\HH_3)^{\Gamma_\8} \arrow[r,"\Eis"] & \Omega^1(\HH_3)^{\Gamma}.
\end{tikzcd}
\end{equation}
Here the vertical arrows are the ismorphisms \eqref{isoform1} and \eqref{isoform2}, and the horizontal map on the bottom is the operator $\Eis$ we considered in \eqref{eisstep1}.

Let $\phi=(\phi_1,\phi_2)$ be as above. We view $\phi$ as a character on $B(\A)$ by
\begin{align}
    \phi \begin{pmatrix}
        a & b \\ 0 & d
    \end{pmatrix}=\phi_1(a)\phi_2(d),
\end{align}
and for a complex parameter $s$ 
\begin{align}
    \phi_s \begin{pmatrix}
        a & b \\ 0 & d
    \end{pmatrix}=\phi_1(a)\phi_2(d)\left \vert \frac{a}{d} \right \vert^{s/2}.
\end{align}
The induced representation $V_\phi$ consists of complex valued functions $\Psi$ on $G(\A_\fin)$ that satisfy $\Psi(bg)=\phi(b)\Psi(g)$ for every $b \in B(\A_\fin)$, and $\Psi(gk)=\Psi(g)$ for every $k \in K_\fin$. Given a function $\Psi \in V_{\phi_s}$, we define 
\begin{align}
    \omegatilde_{\Psi,s} \in \Hom_{K_\8}(\p, C^{\8} \left (B(\Q) \backslash G(\A)/K_\fin \right ))
\end{align}
by
\begin{align}
  \omegatilde_{\Psi,s}[X](g)= \left \vert \frac{a}{d} \right \vert^{1+s}\Psi(g_\fin) \check{S}(\Ad(k_\8^{-1}
  )X) 
\end{align}
where 
\begin{align}
  g_\8=\begin{pmatrix} a & b \\ 0 & d\end{pmatrix}k_\8 \in B(\R)K_\8  
\end{align} and $\check{S} \in \Hom(\p,\C)$ is the dual of $S$. Berger then considers the Eisenstein class
\begin{align}
    \Eis(\Psi,s) \coloneqq \Eis(\omegatilde_{\Psi,s})
\end{align}
where the operator $\Eis$ on the right-hand side is the operator in \eqref{operatoreis}.

Since the class number of $K$ is one, we have a unique equivalence class of cusps and the Eisenstein series that we defined in Section \ref{eisensteincohom} is
\begin{align}
    E_\chi(u,s)=E_{\chi,\Ocal}(u,s)= \Eis(v^s dz) = \sum_{\gamma \in \Gamma_\8 \backslash \Gamma} \gamma^\ast \left( v^s dz\right).\end{align}

\begin{prop} \label{compeis} Let $\chi$ be an unramified character on $\A_K$ such that $\chi_\8(z)=z^2$. Let $\phi_1$ and $\phi_2$ be two characters as above and such that $\chi=\phi_1/\phi_2$. For any $\Psi$ in $V_{\phi}$ we have $\Eis(\Psi,s)=\frac{1}{2}\Eis_{\chi}(u,s)$
\end{prop}
\begin{proof}
    Let $\omega$ be the form in $\Omega^1(\HH_3)^\Gamma$ corresponding to $\omegatilde_{\Psi,s}$ by the isomorphism \eqref{isoform2}. For $u=z+jv$ a point on $\HH_3$, the matrix
    \begin{align}
        b_\8 \coloneqq \begin{pmatrix}
            v & z \\ 0 & 1
        \end{pmatrix} \in B(\R)
    \end{align}
    sends $j$ to $u$. By \eqref{isoadel1}, we then have
    \begin{align}
        \omega_u(X_u) & =\omegatilde_{\Psi,s}[d_uL_{b_\8^{-1}}X_u](b_\8,1) \nonumber \\
        & = v^{1+s} \check{S}(d_uL_{b_\8^{-1}}X_u).
    \end{align}
It is immediate to check that $dL_{b_\8^{-1}}X_u=\diag(v^{-1},v^{-1},v^{-1})$. Hence 
\begin{align}
    dL_{b_\8^{-1}}\left ( \restr{\frac{\partial}{\partial z}}{u} \right )=v^{-1}\restr{\frac{\partial}{\partial z}}{j},
\end{align}
and we find that $\omega=\frac{1}{2}v^sdz$. The result follows from the commutativity of the diagram \eqref{adiagram}.
\end{proof}
Since $K$ has class number one we have $F=H=K$ and the algebraic $L$-function is $L^{\alg}(\chi \circ \N_{H/K},0)=L^{\alg}(\chi,0)=\frac{1}{2}G_2(L_\Ocal)$ and $L^{\integ}(\chi \circ \N_{H/K},0)=2\sqrt{D}G_2(L_\Ocal)$. Now let $\q$ be a prime ideal of $\Ocal$ and let $\Den(\Lcal_{\Eis,\q})$ be as in \eqref{eisatq}.

\begin{cor}\label{maincor} Let $\q$ be a prime ideal coprime to $2D$, and suppose that $K$ has class number one.
    Then the denominator at $\q$ of the Eisenstein cohomology is exactly
    \[ \Den(\Lcal_{\Eis,\q})=G_2(L_\Ocal)\Ocal_{\q}. \]
\end{cor}
\begin{proof} Since $\q$ is coprime to $2$ and $D$ we have 
\begin{align}
    L^{\alg}(\chi \circ \N_{H/K},0)\Ocal_{\q}=L^{\integ}(\chi \circ \N_{H/K},0)\Ocal_{\q}=G_2(L_\Ocal)\Ocal_{\q}.
\end{align}
From Theorem \ref{mainthm} we know that
    \begin{align}
        G_2(L_\Ocal)\Ocal_{\q} \subset \Den(\Lcal_{\Eis})\Ocal_{\q}.
    \end{align}
On the other hand, by \cite[Theorem.~29]{berger1} there is a $\Psi^0 \in V_\phi$ such that
\begin{align}
\Den(\Eis(\Psi^0)) \subset L^{\alg}(\chi ,0)\Ocal_{\q}=G_2(L_\Ocal) \Ocal_\q.    
\end{align}
The conditions of Berger's theorem are satisfied since $\chi$ is unramified and of type\footnote{It correspond to the type $(-2,0)$ in our notation when viewed as a character on the group of fractional ideals.} $\chi_\8(z)=z^2$ by the final remark in \cite[p.~467]{berger1}. By Proposition \ref{compeis}, we have $\Lcal_{\Eis,\q}=\Eis(\Psi^0) \Ocal_{\q}$ hence
\[\Den(\Lcal_{\Eis,\q})=\Den(\Eis(\Psi^0)).\]
\end{proof}
\begin{rmk}\label{indications}
When the class number is greater than one, then the space $S_{K_\fin}$ is no longer connected and is a disjoint union
\begin{align}
    S_{K_\fin}=\bigsqcup_{[\bfrak] \in H_K} \Gamma(\bfrak)\backslash \HH_3,
\end{align}
where $H_K$ can be represented by ideal classes. Hence our result (Theorem \ref{mainthm}) only gives a bound on the denominator of the connected component corresponding to $\bfrak=\Ocal$:
\begin{align} \label{equality0}
    L^{\integ}(\chi \circ \N_{H/K},0)\Ocal_{\q} \subset \Den(\Lcal_{\Eis}(\Ocal))\Ocal_{\q}, 
\end{align}
where $\Lcal_{\Eis}(\bfrak)$ (respectively $\Lcal_{\Eis}(S_{K_\fin})$) is the integral structure on $\Gamma(\bfrak)\backslash \HH_3$ (respectively on $S_{K_\fin}$). In particular, since $\Lcal_{\Eis}(S_{K_\fin})=\bigoplus_{[\bfrak]} \Lcal_{\Eis}(\bfrak)$, we have that
\begin{align}\label{intersection}
 \Den(\Lcal_{\Eis}(S_{K_\fin}))\Ocal_{\q}=\bigcap \Den(\Lcal_{\Eis}(\bfrak))\Ocal_{\q}.
\end{align}
To find the denominator for the connected component associated to the fractional ideal $\bfrak$, we need to consider the Eisenstein series
\begin{align}
    \Ehat_{\chi,\afrak,\bfrak}(u,s)= \chi(\afrak)^{-1}\N(\afrak)^s\sum_{(c,d)\in \afrak \times \afrak\bfrak} \eta(u,c,d,s) \in \Omega^1(\HH_3)^{\Gamma(\bfrak)}.
\end{align}
associated to the cusp $\afrak$. Similary we define $E_{\chi,\afrak,\bfrak}(u,s)$, which is related to $\Ehat_{\chi,\afrak,\bfrak}(u,s)$ by the matrix $M_\chi$ exactly as in \eqref{equality}. We can then define a cocycle
$\Phi_{\afrak,\bfrak} \colon \Gamma(\bfrak)\longrightarrow \C$ by \begin{align}
    \Phi_{\afrak,\bfrak}(\gamma) \coloneqq \chi(\afrak) \int_{u_0}^{\gamma u_0} \Ehat_{\chi,\afrak,\bfrak},
\end{align}
that specializes to Sczech's cocycle when $\bfrak=\Ocal$. We do not know how prove that $\Phi_{\afrak,\bfrak}$ is integral. If we were able to show the integrality at $\q$, we would get the lower bound
\begin{align}
    L^{\integ}(\chi \circ \N_{H/K},0)\Ocal_{\q} \subset \Den(\Lcal_{\Eis}(\bfrak))\Ocal_{\q}
\end{align}
on each of the connected components of $S_{K_\fin}$ by the same argument as in the proof of Theorem \ref{mainthm}. By \eqref{intersection} we would get
\begin{align}
    L^{\integ}(\chi \circ \N_{H/K},0)\Ocal_{\q} \subset \Den(\Lcal_{\Eis}(S_{K_\fin}))\Ocal_{\q}.
\end{align}
After applying Berger's result one should then get the equality
\begin{align}
    L^{\integ}(\chi \circ \N_{H/K},0)\Ocal_{\q} = \Den(\Lcal_{\Eis}(S_{K_\fin}))\Ocal_{\q}
\end{align}
for imaginary quadratic fields of arbitrary class numbers.
\end{rmk}

\printbibliography
\end{document}